\documentclass{amsart}
\usepackage[utf8x]{inputenc}
\usepackage{amsmath, amsthm, amssymb}
\usepackage[usenames,dvipsnames,svgnames,table]{xcolor}
\usepackage[margin=1in]{geometry}
\usepackage{enumerate}
\usepackage{dsfont}
\numberwithin{equation}{section}
\usepackage[colorlinks,linkcolor = blue,citecolor = ForestGreen]{hyperref}
\usepackage{zref-clever}
	\newcommand{\Cref}{\zcref[S]}
\usepackage{cite}
\usepackage{mathtools}
\usepackage[normalem]{ulem}
\usepackage{comment}
\usepackage{xfrac}
\usepackage{xcolor}
\usepackage{stmaryrd}
\usepackage{cancel}
\usepackage{esint}
\usepackage{thmtools}
\usepackage{bm}

\usepackage{soul}

\usepackage{stmaryrd}

\usepackage[title]{appendix}

\newtheorem{theorem}{Theorem}[section]
\newtheorem{proposition}[theorem]{Proposition}
\newtheorem{lemma}[theorem]{Lemma}
\newtheorem{definition}[theorem]{Definition}

\newtheorem{corollary}[theorem]{Corollary}

\newcommand{\R}{\mathbb R}

\newcommand{\eps}{\varepsilon}

\newcommand{\1}{\mathds{1}}

\renewcommand{\epsilon}{\eps}

\newcommand{\be}{\begin{equation}}
\newcommand{\ee}{\end{equation}}

\newcommand{\norm}[1]{\left\lVert#1\right\rVert}

\title[Convergence rates for FKPP with nonlinear advection]{Convergence rates to traveling waves for an FKPP model with nonlinear advection}
\author{Ryan Patterson}

\address[Ryan Patterson]
    {Department of Mathematics, The University of Maryland, 
    College Park, MD 20742, USA}
    \email{ryankp@umd.edu}

\begin{document}

    \begin{abstract}
	We establish convergence rates of solutions to traveling waves for a reaction-diffusion model with nonlinear advection. The convergence rates are proven using the shape defect function. We introduce a
    novel energy methods-based approach that uses a weighted Nash inequality to obtain $L^\infty$-estimates on the shape defect function.
    In determining the convergence rates, we also establish the front location for the model in a way similar to \cite{giletti2022monostablepulledfrontslogarithmic,shortproof}.
    \end{abstract}
    \maketitle

    \section{Introduction}

    \subsection{The model and main question}
    We consider the long time behavior for solutions to the reaction-diffusion equation with nonlinear advection given by
    \begin{equation}\label{main_eq}
        \begin{cases}
            u_t + \chi(A(u))_x = u_{xx} + u-A(u) \quad
        		& \text{ in } (0,\infty)\times \R,\\
                u(0, x) = u_0(x) &\text{ in } \R,
        \end{cases}
    \end{equation}
    where $\chi \in [0, 1]$ and $A \in C^2([0, 1])$.
    We make the assumptions that
    \begin{equation}\label{A_assumptions}
    	A(0) = 0 = A'(0), \quad A(1) = 1, \quad \text{and} \quad \frac{A(u)}{u} \text{ is increasing and convex}.
	\end{equation}
    A prototypical example is the choice $A(u) = u^2$, in which case we recover the Fisher-KPP equation for $\chi = 0$ and the Burgers-FKPP equation for $\chi >0$; see \cite[Chapter 13.4]{murray} and \cite{Leach_Hanac,an2023pushedpulledpushmipullyufronts} for more information
    on the Burgers-FKPP equation.

    For the initial data, we always assume $u_0 \in [0, 1]$ and $u_0  \not\equiv 0$. 
    Note that $u \equiv 1$ and $u \equiv 0$ are steady state solutions to \eqref{main_eq}.
    In addition, $u \equiv 1$ is a stable equilibrium while $u \equiv 0$ is an unstable
    equilibrium.  One interpretation of~\eqref{main_eq} is that of the population density of a species invading a previously unexplored space.
    We also make
    the following standard assumption: there exists $L_0 \in \R$, so that
    \begin{equation}
        \lim_{x \to - \infty} u_0(x) =1
        \qquad \text{ and } \qquad
        u_0(x) = 0 \quad\text{ if } x \geq L_0. \label{initial_data_assumption}
    \end{equation}
    A reasonable initial data to keep in mind that satisfies \eqref{initial_data_assumption} is    \[u_0 = \1_{\{x \leq 0\}}.\]

    The study of solutions to reaction-diffusion equations similar to \eqref{main_eq} dates back to 1937~\cite{KPP,Fisher}, in which the authors considered the equation
    \begin{equation}\label{FKPP}
        u_t = u_{xx} + f(u)
        	\qquad\text{ in } (0,\infty)\times \R. 
    \end{equation}
    One such admissible nonlinearity that serves as a useful example is $f(u) = u - u^2$.
    These works established the existence of traveling wave solutions, which are solutions of~\eqref{FKPP}
    that take the form
    \begin{equation} \label{tw_assumptions}
        u(t, x) = U_c(x-ct)
		\qquad\text{with}\quad
		U_c(-\infty) = 1, \quad U_c(+\infty) = 0, \quad\text{and}\quad U_c \in (0, 1) \text{ for all } x \in \R.
    \end{equation}
    Observe that traveling waves must satisfy the ODE
    \begin{equation}
        -cU_c' = U_c'' + f(U_c)
    \qquad\text{ in } \R. \label{intro_FKPP_TW_eqn}
    \end{equation}
    It turns out that that there is a minimal speed $c_*>0$ such that
    traveling wave solutions exist if and only if $c \geq c_*$.  In the case of $f(u) = u - u^2$, $c_* = 2$~\cite{KPP,Fisher}.

    In~\cite{KPP}, it was additionally proven that
    solutions to \eqref{FKPP} with sufficiently rapidly decaying initial data, $u_0$, converge in shape
    to $U_{c_*}$. This means that there exists a reference frame $m(t)$, which we refer to as the front location, so that
    \begin{equation}
        |u(t, x + m(t)) - U_{c_*}(x)| \to 0 \quad \text{as } t \to +\infty, \text{ uniformly in } x\in \R. \label{kpp_result}
    \end{equation}
    Convergence in shape was originally proven by a compactness-style argument. As a result, it is not quantitative.

    Our model~\eqref{main_eq} with advection also has traveling waves
    satisfying \eqref{tw_assumptions}, however, the ODE which they satisfy is
    \begin{equation}
        -cU_c' + \chi A'(U_c)U_c'= U_c'' + U_c - A(U_c)
    \qquad\text{ in } \R. \label{intro_TW_eqn}
    \end{equation}
    By~\cite{hadeler_rothe}, there exists a minimal speed $c_* > 0$ so that
    \eqref{main_eq} has a unique (up to translation) traveling wave solution if and only if $c \geq c_*$.
    The authors of~\cite{hadeler_rothe} additionally show that the minimal speed is $c_* = 2$ for all $\chi \in [0, 1]$. We denote 
    \[
    		U_* := U_{c_*} = U_2.
	\] 
    Convergence results for \eqref{main_eq} are unknown except in
    the special case $\chi = 0$ when~\eqref{main_eq} simplifies to the more well-studied form~\eqref{FKPP}. Therefore,
    the goal of this paper is to prove and quantify convergence of $u$ to $U_*$.  We state our main theorem, up to defining ``steep initial data,'' which we do in \Cref{d.steepness}.

    \begin{theorem}\label{main_result_theorem}
        Fix $\chi \in[0,1]$ and $A$ satisfying~\eqref{A_assumptions}.  Let $u$ be a solution to~\eqref{main_eq} with ``steep initial data'' $u_0$ 
        for which \eqref{initial_data_assumption} holds.
        There is a reference frame $m= m_\chi$ such that
        \begin{equation}
            \norm{u(t, \cdot + m(t)) - U_*}_{L^\infty} \leq \frac{C}{t}. \label{main_result}
        \end{equation}
    \end{theorem}
    In proving \Cref{main_result_theorem}, we also find the asymptotics of the front location $m_\chi(t)$ up to $O(1)$ precision, which is 
    established in \Cref{front_location_theorem}. 
    The main ingredient in the proof of \Cref{main_result_theorem} is \Cref{shape_defect_theorem}.  In order to present that result, however, we must develop some notation.  For this reason, we postpone its statement, but alert the reader that the main contributions
    of this paper are discussed following \Cref{shape_defect_theorem}.

    Algebraic convergence rates similar to \eqref{main_result} seem to be quite rare in the reaction-diffusion
    literature. We mention the results of Graham ~\cite{graham} and An, Henderson, and Ryzhik ~\cite{an2023locationdeterminesconvergencerate}.
    Each paper proves $O(1/t)$ convergence rates for different nonlinearities of \eqref{FKPP}.
    More exposition is given about these results following \Cref{shape_defect_theorem}.
    In addition, there recently have been results that
    achieve algebraic convergence rates in settings outside the scope that we consider. For example,
    if one considers initial data that does not satisfy \eqref{initial_data_assumption}, then
    faster than $O(1/t)$ convergence rates can be obtained. This was shown by Avery and Scheel in ~\cite{avery_scheel}.

    To see why an algebraic convergence rate such as  \eqref{main_result} might
    be expected, consider the linearized problem associated with \eqref{main_eq}:
    \begin{equation}
        v_t = v_{xx} + v. \label{intro_linearized}
    \end{equation}
    The spreading speed for \eqref{intro_linearized} is $c_{\text{lin}}=2$. Therefore,
    the minimal speed for \eqref{main_eq} and the linearized spreading speed are the same. 
    This phenomenon is called ``pulled'' front propagation, and
    intuitively means that the linearized problem is a good approximation to \eqref{main_eq}.
    Thus, the portion that ``matters'' is when $u$ is small, or, equivalently, as $x \to +\infty$. Therefore, the convergence is
    determined over an unbounded set, which results in algebraic decay rates.
    Compare, for example, the heat equation on a compact set versus over $\R$. The first converges exponentially
    while the latter converges like $t^{-1/2}$.

    \subsection{The effect of $\chi$} \label{intro_front_location}
    Understanding how $\chi$ affects solutions of \eqref{main_eq} is
    an important piece in proving convergence. In particular, the behavior in the
    case $\chi \in [0, 1)$ is quite different than when $\chi = 1$.
    The first example of such a difference is the
    following result on the front location.
    \begin{proposition}\label{front_location_theorem}
        Suppose that $u$ solves \eqref{main_eq} with the initial data assumptions given in \eqref{initial_data_assumption}.
        \begin{enumerate}[(i)]
            \item\label{i.pulled_front}
            If $0 \leq \chi < 1$, then
            \begin{equation}
                m(t) = 2t - \frac{3}{2}\log t + O(1). \label{pulled_front}
            \end{equation}
            \item\label{i.pushmi-pullyu_front}
            	 If $\chi = 1$, then
            \begin{equation}
                m(t) = 2t - \frac{1}{2}\log t + O(1).\label{pushmi_pullyu_front}
            \end{equation}
        \end{enumerate}
    \end{proposition}

    \Cref{front_location_theorem} was known in certain special cases.  First, when $\chi = 0$, there is no advection and it goes back to Bramson~\cite{Bramson1,Bramson2}.
    For $\chi = 1$, the result is a consequence of the main theorem of An, Henderson, and Ryzhik \cite{an2022quantitativesteepnesssemifkppreactions} using a miraculous relationship between~\eqref{main_eq} and~\eqref{FKPP} with $f(u) = (u-A(u))(1+ A'(u))$.
	For general $\chi$, to our knowledge, the only previous result on the front location is for the special case $A(u) = u^2$.
	Here, the front location was established by Leach and Hana\c{c}~\cite{Leach_Hanac}, formally, and then by An, Henderson, and Ryzhik~\cite{an2023pushedpulledpushmipullyufronts}, rigorously.
    We also mention the recent work~\cite{demircigil2026locationgrowmodelsaerotaxis} of Demircigil and Henderson, who proved front location asymptotics for a degenerate model with a discontinuous nonlinearity $A$.  Note that this falls outside of the assumptions of our work.
    This model is related to a free boundary problem studied by
    Berestycki, Penington, and Tough in \cite{berestycki2025convergencepositionfkpptypefree} in which
    the same front location was proven; see also~\cite{berestycki_free_boundary}.

    The proof of \Cref{front_location_theorem} when $\chi \in (0,1)$  is similar to that of
    Hamel, Nolen, Roquejoffre, and Ryzhik ~\cite{shortproof}, who considered the equation~\eqref{FKPP}
    with $f(u)$ under the ``KPP condition'' $f(u) \leq f'(0) u$.
    They construct a supersolution using the linearized Dirichlet problem ahead of the front.
    We combine the ahead-of-front supersolution with the idea of Giletti in~\cite{giletti2022monostablepulledfrontslogarithmic} to match the supersolution
    with a traveling wave behind the front.
    Giletti used this to generalize the work of~\cite{shortproof} on the equation~\eqref{FKPP} to the broader class of nonlinearities $f$ without the KPP condition.  We use these ideas to obtain an upper bound on the front location.
    The lower bound is nearly immediate by comparing \eqref{main_eq}
    to the reaction-diffusion equation
    \[u_t = u_{xx} + u - A(u).\]
    Then, from the above discussion, we already have the lower bound on the front location:
    \[m(t) \geq 2t - \frac{3}{2}\log t + O(1).\]
    The detailed proof of \Cref{front_location_theorem} can be found in \Cref{front_location_section}.

    At first glance, this transition in \Cref{front_location_theorem} may seem mysterious. Therefore, we would like to give some informal ideas for why we see such a change in $m(t)$.
    Consider
    \begin{equation}
    		\bar u(t, x) = e^x \tilde u(t, x) = e^x u(t, x + 2t).
	\end{equation}
    By~\eqref{main_eq}, we find
    \begin{equation}
        \bar u_t = \bar u_{xx} - \chi(e^x A(\tilde u))_x -(1-\chi)e^{x}A(\tilde u). \label{intuition_eqn}
    \end{equation}
    When $\chi = 1$, we see that $\norm{\bar u}_{L^1}$ is conserved, while
    when $\chi < 1$, $\norm{\bar u}_{L^1}$ is decreasing. Very informally, if we
    ignore the advective term in \eqref{intuition_eqn}, then for $\chi = 1$,
    we have the heat equation. Then solutions would decay like $\frac{1}{\sqrt t}$.
    In other words,
    \[u(t, x + 2t) \sim \frac{Ce^{-x}}{\sqrt t}.\]
    Then to ensure $u(t, m(t))$ stays bounded away from $0$, one must shift by $\frac{-1}{2} \log t$.

    If $\chi < 1$, then we have a negative nonlinear term: 
    \[
        -(1-\chi) \alpha(\tilde u) \bar u.
    \]
    Here, $\alpha(u) = A(u)/u$, and behind the front $\alpha(u) \sim 1$.
    Therefore, we have what looks like exponential decay behind the front. This
    makes the solution behave like that of a linearized Dirichlet problem. Hence,
    from the ideas of \cite{shortproof}, we would then expect a $-\frac{3}{2}\log t$ shift.
    Of course, this hinges on the believing that we can ignore the $(e^x A(\tilde u))_x$ term, which is not
    so obvious.

    Besides the front location, there is another key difference between the two cases:
    the asymptotics of the minimal speed traveling wave.
    When $\chi \in[0,1)$, we have exponential decay with a linear multiplicative factor: there is $\theta>0$ such that
        \begin{equation}
            U_*(x) = \theta xe^{-x} + O(e^{-x}) \qquad \text{as } x \to +\infty. \label{pulled_asympotics}
        \end{equation}
        On the other hand, in the case $\chi = 1$, $U_*$ has purely exponential asymptotics: there is $\theta>0$ such that
        \begin{equation}
            U_*(x) = \theta e^{-x} + o(e^{-x}) \qquad \text{as } x \to +\infty. \label{pushmi_pullyu_asympotics}
        \end{equation}

    The properties in the case $\chi = 1$, \eqref{pushmi_pullyu_front} and \eqref{pushmi_pullyu_asympotics}, are not typically seen in pulled fronts.
    For this reason, following the nomenclature in ~\cite{an2023pushedpulledpushmipullyufronts},
    we call the $\chi = 1$ case ``pushmi-pullyu.'' On the
    other hand, pulled fronts, generally, have the properties \eqref{pulled_front} and \eqref{pulled_asympotics}
    so we call the $\chi < 1$ case ``pulled.''
    Although a similar overarching approach is used in proving \Cref{main_result_theorem} for both cases,
    the analysis between the two is significantly different. This manifests in the
    pushmi-pullyu case being a fair amount simpler than the pulled case;
    see, for example \eqref{intro_pushmi_pullyu_eta} and the subsequent discussion.

    Before we move on, let us mention that the case $\chi > 1$ gives a ``pushed'' front.
    The methods for determining convergence for pushed fronts tend to
    be different because they yield exponential convergence rates \cite{rothe,FifeMcleod}.
    We would be interested to see if it is possible to adapt our weighted functional
    inequality method to the pushed case.

    \subsection{Shape defect function} \label{intro.ss.shape_defect}
    The main object we use to study convergence rates is the shape defect function,
    which was first introduced in \cite{an2022quantitativesteepnesssemifkppreactions}. We give a brief overview of its construction and properties.
    
    Since traveling wave solutions
    are monotonically decreasing, there exists a $C^1(0, 1)$ function $\eta(u)$ so that
    \begin{equation}
        -U_*' = \eta(U_*). \label{profile_function}
    \end{equation}
    The function $\eta$ is called the traveling wave profile function. In the
    pushmi-pullyu case $\chi = 1$, $\eta$ can be written explicitly as
    \begin{equation}
        \eta(u) = u - A(u). \label{intro_pushmi_pullyu_eta}
    \end{equation}
    When $\chi \in [0, 1)$, however, there is no such explicit formula. This fact
    causes significant complications in the upcoming proofs because in the
    pulled case we must work with asymptotic expansions of $\eta$ and
    its derivatives.

    The shape defect function is then defined to be
    \begin{equation}
        w(t, x) = -u_x - \eta(u(t, x)). \label{shape_defect_function}
    \end{equation}
    Recalling \eqref{profile_function}, $w$ compares the slope of $u$ to the slope of $U_*$ at $U_*^{-1}(u(t, \cdot))$.
    Therefore, $w$ gives us a way to quantitatively compare the steepness of the traveling
    wave to $u$, and allows us to define ``steep'' initial data in \Cref{main_result_theorem}, which we do now.

    \begin{definition}\label{d.steepness}
        We say a function $\mu: \R \to [0,1]$ is
        steep if
        \begin{equation}
            w_\mu := -\mu_x - \eta(\mu) \geq 0.
        \end{equation} 
    \end{definition}
    Having steep initial data means that $u_0$ is steeper than
    the traveling wave $U_*$. In addition, although it is not yet obvious, we see below that $u(t, \cdot)$ remains steep for $t > 0$ if $u_0$ is steep.
    This can be seen by considering the PDE \eqref{shape_defect_evolution_eqn} for $w$ and applying the comparison principle.

    The shape defect function also has the remarkable property
    that its decay can be leveraged to show that $u$ and $U_*$ are close together.
    In other words, one can show, roughly, that
    \begin{equation}
    		|u(t, x + m(t)) - U_*(x)| \sim \|w\|. \label{rough_idea}
    \end{equation} 
    The idea behind \eqref{rough_idea} is the simple observation, developed in \cite[Theorem 2.1]{an2023locationdeterminesconvergencerate}, that
    \begin{equation}
        s(t, x) := u(t, x+m(t)) - U_*(x)
    \end{equation}
    satisfies for some $\xi$ between $u(t, x+m(t))$ and $U_*(x)$,
    \begin{equation}
        s_x = u_x - U_*' = - w - \eta(u) + \eta(U_*) = - w - \eta'(\xi)s. \label{intro_difference}
    \end{equation}
    One can apply the variation of parameters formula to this ODE and use the decay on $w$ along with asymptotics for $\eta$
    to obtain a decay rate for $s$. The asymptotics for $\eta$ are easily proven from the
    asymptotics of $U_*$, see \cite[Lemma 3.4]{an2023locationdeterminesconvergencerate}. 
    Hence, the difficult part of this argument is in proving the following (reasonably) sharp bounds on $w$.

    \begin{theorem}\label{shape_defect_theorem}
        Suppose that $u$ solves \eqref{main_eq} with the same assumptions as \Cref{main_result_theorem}. Then
        \begin{enumerate}[(i)]
            \item\label{i.pulled}
            If $0 \leq \chi < 1$, then
            \begin{equation}
                w(t, x + m(t)) \leq \frac{C}{t} \1_{\{x \leq 0\}} + \frac{C(1+x)^2e^{-x}}{t} \1_{\{x \geq 0\}}. \label{pulled_w_decay}
            \end{equation}
            \item\label{i.pushmi-pullyu}
            	 If $\chi = 1$, then
            \begin{equation}
                w(t, x+m(t)) \leq \frac{C}{t} \1_{\{x \leq 0\}} + \frac{C(1+x)e^{-x}}{t} \1_{\{x \geq 0\}}. \label{pushmi_pullyu_w_decay}
            \end{equation}
        \end{enumerate}
    \end{theorem}
    
        In fact, one can use the argument outlined above to deduce \Cref{main_result_theorem} from \Cref{shape_defect_theorem}.  We omit the details, which follow~\cite[Theorem 2.1]{an2023locationdeterminesconvergencerate} exactly, and only prove \Cref{shape_defect_theorem} in this work.

    The most recent work related to ours is~\cite{an2023locationdeterminesconvergencerate}, which considers the pure reaction-diffusion equation~\eqref{FKPP} and not our nonlinear advection model~\eqref{main_eq}.  Here, the authors noticed the link between the shape defect function and convergence rates to the traveling wave.  
    They use a complex barrier argument that implicitly uses a fine understanding of the shape of $w$ in order to bound $w$ pointwise and deduce convergence rates.
    
    In contrast, our proof of \Cref{shape_defect_theorem} for the nonlinear advection model~\eqref{main_eq} follows ideas dating back to Nash~\cite{Nash} to deduce pointwise bounds in $w$ from simple functional inequalities.  
    The result is a cleaner and possibly more general proof, which we believe applies equally to~\eqref{FKPP}.
    We give a more detailed exposition in \Cref{intro_key_elems} on our approach; that is, how one deduces \Cref{shape_defect_theorem} from a weighted Nash inequality.

    Before~\cite{an2023locationdeterminesconvergencerate},
    the front location for the Fisher-KPP equation was studied extensively
    and led to a proof of sharp convergence rates.  Let us emphasize that this proof applies to the special case of the pure reaction-diffusion equation~\eqref{FKPP} under the assumption that $f(u) \leq f'(0) u$;
    e.g., $f(u) = u(1-u)$.  In a sense, one should think of this as analogous to our case with ``small'' $\chi$. 
    
    Let us illustrate why precise front location asymptotics is a bottleneck for convergence rates with the approach of~\cite{shortproof, nolen2016convergencesinglewavefisherkpp, nolen2018refinedlongtimeasymptotics,graham}.  Very roughly, 
    if an approximation, $m_{\rm app}$, ``misses'' the true front location $m(t)$
    by an error rate $O(\epsilon_t)$, then
    $\|u(t,\cdot + m_{\rm app}) - U_*\|$ is also $O(\epsilon_t)$.  Indeed, even if convergence was exact, i.e., $u(t,\cdot + m) = U_*$, we find
    \begin{equation}
		u(t,\cdot + m_{\rm app}) - U_*
			= u(t,\cdot + m_{\rm app}) - u(t,\cdot + m)
			\approx u_x(t,\cdot + m_{\rm app}) (m_{\rm app} - m)
			= O(\eps_t).
    \end{equation}
    It is for this reason, that, in successive works, authors sought more and more precise expansions of $m$.

    The effort to understand the front location gave rise to the program of Hamel, Nolen, Roquejoffre, and Ryzhik~\cite{shortproof, nolen2016convergencesinglewavefisherkpp, nolen2018refinedlongtimeasymptotics}
    in which each paper proves successively finer asymptotics.  The last of these papers showed
    \begin{equation}
        m(t) = 2t - \frac{3}{2}\log t + x_0 - \frac{3\sqrt \pi}{\sqrt t} + O(t^{-1 + \gamma}) \quad \text{ as } t \to +\infty,
    \end{equation}
    for $x_0$ depending on the initial data and any $\gamma >0$. As discussed above, this yields a convergence rate of $O(t^{-1+\gamma})$, which is not quite the sharp $O(1/t)$ asymptotics.  Afterwards, Graham \cite{graham}
    capped off the program by showing that
    \begin{equation}\label{fine_asymp}
        m(t) = 2t - \frac{3}{2}\log t + x_0 - \frac{3\sqrt \pi}{\sqrt t} + \frac{9}{8}(5-6\log 2)\frac{\log t}{t} + O(1/t) \quad \text{ as } t \to +\infty.
    \end{equation}
    This leads to $O(1/t)$ convergence rates for the Fisher-KPP equation.  Additionally, \cite{graham} also shows that
    $O(1/t)$ is the best convergence rate in this setting.
    Note that this result applies to the pure reaction-diffusion equation~\eqref{FKPP}; that is, it does not apply to our nonlinear advection model \eqref{main_eq} except in the special case $\chi = 0$.

    The asymptotics in \eqref{fine_asymp} had been
    expected due to several earlier predictions.
    The formal understanding of~\eqref{fine_asymp} goes back at least to
    Ebert and van Saarloos~\cite{Ebert_2000}, who predicted the $3\sqrt \pi$ coefficient
    in \eqref{fine_asymp}, while Berestycki and Brunet \cite{berestycki2016noteconvergencefisherkppcentred} found that {\em if} $m'(t) = 2 + O(1/t)$,
    {\em then} $u$ converges to $U_*$ with $O(1/t)$ correction. Later,
    Berestycki, Brunet, and Derrida \cite{new_approach} predicted the $\frac{\log t}{t}$ term and its
    coefficient.
    Similar front location asymptotics have been explored in other contexts and generalizations; we refer the reader to
    \cite{henderson,vanishing_corrections, Berestycki_2017,giletti2022monostablepulledfrontslogarithmic, avery_scheel_selection} for more details.

    In the end, the $O(1/t)$ convergence rate was predicted and proven from
    the difficult front location calculations in the papers mentioned above.
    Our approach with the shape defect function allows us to sidestep the fine front location asymptotics.  Indeed, one does not need a precise understanding of $m$ to deduce~\eqref{rough_idea}, which is the basis for our arguments.  Moreover, bounds on $\|w\|$ in the $m$ and $m_{\rm app}$ frames are equivalent up to a multiplicative constant as long as $m = m_{\rm app} + O(1)$.
    This simplifies the proof of convergence rates in a major way because the
    most we need is $O(1)$ front location precision. For this reason,
    the asymptotics in \Cref{front_location_theorem} are sufficient for our purposes.

    \subsection{Key elements of the proof of \Cref{shape_defect_theorem}}\label{intro_key_elems}

    Now we give an overview for the main difficulties in proving \Cref{shape_defect_theorem},
    along with the ideas we use to overcome them. The detailed proof
    can be found in \Cref{pushmi_pullyu_section,pulled_section}.

    \subsubsection{Shape defect function evolution equation}\label{evolution_eqn_section}
    We primarily analyze $w$ through the partial differential equation that it satisfies. In \Cref{derivation},
    we compute the evolution equation for $w$:
    \begin{equation}
        w_t = w_{xx} - \chi A'(u)w_x + (1-A'(u))w + 2w(w+\eta(u))(\eta''(u) + \chi A''(u)). \label{shape_defect_evolution_eqn}
    \end{equation}
    At first, \eqref{shape_defect_evolution_eqn} may seem difficult to handle, however, there are a few facts that we point out. 
    First, note that $w = 0$ is a solution, so by the comparison principle, if
    $w_0(x) = w(0, x) \geq 0$, then $w(t, x) \geq 0$ for all $t > 0$. An immediate
    consequence is that if $w_0 \geq 0$, then $u_x \leq 0$ for all time.
    
    The second fact we point out, which is not immediately obvious, is that for all $\chi \geq 0$,
    \begin{equation}
        \eta''(u) + \chi A''(u) \leq 0. \label{bad_nonlinear_term}
    \end{equation}
    For $\chi = 1$, recalling \eqref{intro_pushmi_pullyu_eta}, we have that
    \[\eta''(u) + A''(u) = (u - A(u))'' + A''(u) = 0.\]
    For $\chi \in [0, 1)$,
    we prove this fact in \Cref{concave}.
    Therefore, being slightly heavy handed, we can ignore $\eta'' + \chi A''$ in~\eqref{shape_defect_evolution_eqn}, and just consider
    \begin{equation}
        w_t \leq w_{xx} - \chi A'(u)w_x + (1-A'(u))w. \label{simplified_eqn}
    \end{equation}

    The last fact we mention is that $1-A'(u)$ is negative ``far enough''
    behind the front; that is, when $u$ is close to $1$. For example, take $A(u) = u^2$, so that
    \[1-A'(u) = 1 - 2u.\]
    In this case, whenever $u > 1/2$, we have $1-A'(u) < 0$.
    Using \eqref{simplified_eqn}, we see that any growth of $w$ can only occur ahead of the front.
    This gets at the crux of of our arguments
    in which we must balance the growth ahead of the front with the decay behind the front.

    \subsubsection{Digression to heat equation} \label{toy_model}
    To motivate our approach, let us investigate the heat equation, which is a toy model
    we appeal to frequently.
    Let $z_t = z_{xx}$ with
    initial data $z_0(x) \geq 0$ decaying quickly enough. Then the Nash inequality, along with the fact that $\norm{z}_{L^1}$ is conserved, implies that
    \begin{equation}
        \frac{1}{2}\frac{d}{dt}\left[\int z^2 dx\right] = -\int z_x^2 dx \leq -C \frac{\left(\int z^2 dx\right)^{3}}{\left(\int z_0 dx\right)^4}. \label{heat_eqn_dissipation}
    \end{equation}
    Thus,
    \begin{equation}
        \norm{z(t)}_{L^2} \leq \frac{C\norm{z_0}_{L^1}}{t^{1/4}}. \label{heat_eqn_L1_L2}
    \end{equation}
    By a standard adjoint argument using the symmetry of the Laplacian, one can also show
    \[\norm{z(t)}_{L^\infty} \leq \frac{C\norm{z_0}_{L^2}}{t^{1/4}}.\]
    Combining these bounds, one obtains
    \begin{equation}
        \norm{z(t)}_{L^\infty}
    		\leq \frac{C \norm{z(t/2)}_{L^2}}{(t/2)^{1/4}}
		\leq \frac{C \norm{z_0}_{L^1}}{(t/2)^{1/2}}. \label{heat_eqn_L2_Linf}
    \end{equation}
    Therefore, we see that an $L^2$-based energy method with the Nash inequality yields the
    correct temporal decay for solutions to the heat equation.

    Although the equation~\eqref{shape_defect_evolution_eqn} for $w$ does not initially
    appear similar to the heat equation,
    after the transformation
    \[\bar w(t, x) = e^x \tilde w(t, x) = e^x w(t, x+2t),\]
     we have
     \begin{equation}
        \bar w_t = \bar w_{xx} - \chi A'(\tilde u) \bar w_x - (1-\chi) A'(\tilde u) \bar w + 2\bar w(\tilde w + \eta)(\eta'' + A'').
        \label{intro_bar_w}
     \end{equation}
    The pushmi-pullyu case ($\chi = 1$) provides a good example for intuition since \eqref{intro_bar_w} simplifies to
    \begin{equation} \label{intro_transformed}
        \bar w_t = \bar w_{xx} - A'(\tilde u) \bar w_x.
    \end{equation}
    This looks much more like the heat equation, so it is plausible that an $L^2$-energy argument
    using a Nash-type inequality could be successful.

    \subsubsection{Weighted Nash inequality}
    We now state our main functional inequality. In the upcoming proofs of \Cref{shape_defect_theorem}.\eqref{i.pushmi-pullyu}
    and \Cref{shape_defect_theorem}.\eqref{i.pulled}, we only use $\beta = 1$ and $\beta = 2$, respectively. There is no
    extra difficulty in proving \Cref{weighted_nash_ineq_lemma} for all $\beta \geq 0$, so we give the general statement.
    \begin{proposition}\label{weighted_nash_ineq_lemma}
        Let $\beta \geq 0$. There exists a constant $C > 0$, depending only on $\beta$, such that, for all $\varphi \in H^1(\R_+)\cap L^1(\R_+, x^\beta dx)$ and all $L > 0$,
        \begin{equation}
            \int_0^\infty \varphi^2dx \leq CL^2\int_0^\infty \varphi_x^2 dx + \frac{C}{L^{2\beta+1}}\left( \int_0^\infty |\varphi|x^\beta dx\right)^2. \label{weighted_nash_ineq}
        \end{equation}
    \end{proposition}
    A reader well-versed in functional inequalities might recognize that
    \eqref{weighted_nash_ineq} is not how the Nash inequality is usually stated, however, after setting $\beta = 0$ and
    choosing $L$ to optimize the right side of \eqref{weighted_nash_ineq}, we recover the usual form. It is beneficial to
    leave the inequality as \eqref{weighted_nash_ineq} for flexibility in future calculations.
    The proof of \Cref{weighted_nash_ineq_lemma}
    can be found in \Cref{weighted_nash_ineq_proof_section}.

    Notice that the integrals in \eqref{weighted_nash_ineq} have bounds only on the positive half line.
    This poses a slight problem, because we make no assumptions on the value of $w$ at $x=0$.
    Fortunately, due to the negativity of $1-A'(\tilde u)$ behind the front, we are able to ``match'' the
    estimate \eqref{weighted_nash_ineq} on the right with a favorable estimate on the left. We then get a dissipation estimate
    that is similar to \eqref{heat_eqn_dissipation}.

    From \eqref{weighted_nash_ineq}, we see that weighted $L^p$ spaces naturally appear.
    Below in \Cref{notation_section}, we define the weighted $L^p_g$ spaces that are used in
    this paper.

    \subsubsection{Adjoint argument}\label{intro_adjoint_arg}
    Let us return to our toy model, the heat equation. The symmetry that we mention
    in \Cref{toy_model} is that the formal adjoint equation for the heat equation is
    still the heat equation.
    This fact is convenient in that is allows an $L^1 \to L^2$ bound to be
    bumped up to an $L^1 \to L^\infty$ bound.
    Unfortunately, we do not have this luxury when we consider the shape defect function.
    To see this, recall the transformed equation ~\eqref{intro_transformed} when $\chi = 1$:
    \begin{equation}
                \bar w_t = \bar w_{xx} - \chi A'(\tilde u)\bar w_x. \label{intro_w_2}
    \end{equation}
    Therefore, the formal adjoint equation to 
    \eqref{intro_w_2} is given by
    \begin{equation}
        v_t = v_{xx} + \chi (A'(\tilde u) v)_x. \label{intro_adjoint}
    \end{equation}
    Since \eqref{intro_w_2} is not the same as \eqref{intro_adjoint},
    we cannot immediately deduce an $L^2 \to L^\infty$ bound from an $L^1 \to L^2$ bound.

    The way we get around this is by considering the solution operator $\tilde S_t$ for \eqref{intro_adjoint}.
    This solution operator is defined as the mapping of an initial data $v_0$ to the solution of \eqref{intro_adjoint}
    at time $t$. If we can compute an $L^1_g \to L^2$ bound for $\tilde S_t$, we then obtain 
    the same $L^2 \to L^\infty_{1/g}$ bound for the adjoint operator of $\tilde S_t$, denoted $\tilde S_t^*$.
    The key point is that $\tilde S_t^*$ is a solution operator
    of \eqref{intro_w_2}. Therefore, we can
    still have an analogue of
    \eqref{heat_eqn_L2_Linf} even for our equation that lacks the symmetry of the heat equation. There is a catch, though. We must
    apply two $L^2$-energy arguments: one for the forward equation and one for the formal adjoint equation.

    One of the quirks of this adjoint argument is that only the weight
    used in analyzing the adjoint equation appears in the $L^\infty_{1/g}$-bound.
    We use this fact in our
    arguments, as we do not compute a weighted $L^1 \to L^2$ bound for \eqref{intro_w_2}.
    Instead we directly compute the $L^1$-norm, using the definition of the shape
    defect function \eqref{shape_defect_function}, and then using pointwise bounds on $u$.
    That is, we calculate
    \begin{equation}
        \int \bar w dx = \int e^x \tilde w(t, x)dx = \int e^x (-\tilde u_x-\eta(\tilde u))dx = \int e^x (\tilde u - \eta(\tilde u))dx.
    \end{equation}
    In the last equality, we integrated by parts.  Therefore, knowing estimates on $u$ and $\eta$ allows us to calculate $\norm{\bar w}_{L^1}$.

    Of course, this argument immediately fails for the formal adjoint equation because its
    solution, $v$,
    almost certainly is not a shape defect function coming from a solution to \eqref{main_eq}.
    In other words, we cannot write 
    \[v = e^x(-\tilde u_x - \eta(\tilde u))\]
    like we could for $\bar w$.
    Therefore, to bound $\norm{v}_{L^1_g}$, we must honestly compute estimates on
    \begin{equation} \label{intro_adjoint_L1}
        \frac{d}{dt} \int v g(x)dx
    \end{equation}
    using only the formal adjoint equation. Recall that as $x \to +\infty$, $g(x) \sim x^\beta$
    where $\beta = 1$ if $\chi = 1$ and $\beta = 2$ if $\chi < 1$. The reason for the different choice of $\beta$ is
    given below in \Cref{correct_weight_into}. This difference in weights also causes the
    main technical difference between the pulled ($\chi < 1$) and pushmi-pullyu ($\chi = 1$) cases.
    Notice that when we calculate \eqref{intro_adjoint_L1}, the Laplacian term disappears
    after integrating by parts twice when $\chi = 1$ since $g(x)$ is linear. 
    On the other hand, for the pulled case,
    the weight is quadratic meaning
    that the Laplacian term remains:
    \begin{equation} \label{remaining_term}
        \int \bar w_{xx} g(x)dx \approx 2\int \bar w dx.
    \end{equation}
    Compensating for \eqref{remaining_term}
    can only be done through by utilizing the negative integral
    \begin{equation} \label{compensating_term}
        \int \bar w \eta(\tilde u)(\eta''(\tilde u) + \chi A''(\tilde u)) g(x)dx,
    \end{equation}
    and the cancellation between \eqref{remaining_term} and \eqref{compensating_term}
    turns out to be quite delicate.

    \subsubsection{Choosing the correct weight} \label{correct_weight_into}
    To extract the sharp decay for $w$, we must be careful in choosing our weights.
    Essentially, for
    $\chi = 1$, the weight we choose is $g(x) \sim x$ as $x \to \infty$, while for
    $\chi \in [0, 1)$, we choose $g(x) \sim x^2$ as $x \to \infty$. It is not a coincidence
    that the weights match with the estimates on $w$ in \eqref{pulled_w_decay} and \eqref{pushmi_pullyu_w_decay}.  This is discussed 
    above in \Cref{intro_adjoint_arg}.
    
    To see why it makes sense to choose these weights, first consider the equation for $w$ in the pushmi-pullyu case.
    Since $\eta'' + A'' = 0$, \eqref{shape_defect_evolution_eqn} simplifies to
    \[w_t = w_{xx} + w(1-A'(u)) - A'(u) w_x.\]
    Notice that ahead of the front, this equation has the linearization
    \begin{equation}
        w_t = w_{xx} + w. \label{intro_w_lin}
    \end{equation}
    The linearized equation \eqref{intro_w_lin} is the same as the linearization of the Fisher-KPP equation, \eqref{intro_linearized}. 
    In the front shifted reference frame, solutions to Fisher-KPP look like
    \begin{equation}
        xe^{-x}e^{-\frac{x^2}{4t}} \quad \text{ as } x \to \infty. \label{intro_w_guess}
    \end{equation}
    Recall that from \Cref{intro_adjoint_arg}, an $L^1_g \to L^2$ bound for the formal adjoint equation
    yields a $L^2 \to L^\infty_{1/g}$ bound for the shape defect function. To
    match with the algebraic factor $x$ in \eqref{intro_w_guess}, we must choose $g(x)$
    to also be linear ahead of the front.

    For the pulled case, the intuition for the quadratic weight is not as
    easy to see. The difficulty is caused by the negative nonlinear term:
    \[2w\eta(u)(\eta''(u) + \chi A''(u)).\]
    For $u$ small, this term looks like
    \begin{equation}
        2w\eta(u)(\eta''(u) + \chi A''(u)) \sim -\frac{2w}{(\log u)^2}. \label{log_approx}
    \end{equation}
    Nonlinearities $f$ in~\eqref{FKPP} of the type
    \begin{equation} \label{bouin_nonlinearity}
        f(u) \approx u -\frac{\kappa u}{(\log u)^2}
        \qquad\text{ when } u \ll 1
    \end{equation}
    have been studied by Bouin and Henderson \cite{bouin2020bramsondelayfisherkppequation}. 
    Assuming that \eqref{log_approx} can be understood in the same way as \eqref{bouin_nonlinearity},
    their work suggests that
    $w$ should look, roughly, like
    \begin{equation} \label{intro_pulled_w_guess}
        w\Big( t, x + 2t - \frac{5}{2}\log t\Big) \sim x^2 e^{-x}e^{-\frac{x^2}{Ct}} \quad \text{as } x \to +\infty.
    \end{equation}
    Importantly,  notice the $2t - (\sfrac{5}{2}) \log t$ moving frame in contrast to \Cref{front_location_theorem}.(i).
    In order to capture the quadratic prefactor in \eqref{intro_pulled_w_guess}, the weight $g(x)$
    must also be quadratic as $x \to \infty$.

	\subsection{Organization of the paper}
    In \Cref{pushmi_pullyu_section, pulled_section}, we prove
    \Cref{shape_defect_theorem}.\eqref{i.pushmi-pullyu} and \Cref{shape_defect_theorem}.\eqref{i.pulled} respectively.
    Their proofs are similar in most aspects but diverge greatly in calculating $L^1$-estimates. \Cref{pulled_section}
    also establishes the required properties of the traveling wave profile function $\eta$.
    \Cref{weighted_nash_ineq_proof_section} gives the proof of the weighted Nash inequality: \Cref{weighted_nash_ineq_lemma}.
    The proof is quite simple, but also gives a flexible form of the Nash inequality, which we make use of in
    \Cref{pushmi_pullyu_section,pulled_section}.
    \Cref{front_location_section} contains the proof of the front location up to $O(1)$ precision
    when $\chi \in [0, 1)$. The first part of the section gives a supersolution for the right half line
    and the second part of the section gives a supersolution for the left half line.

    Finally, we prove several miscellaneous results in the appendix.
    In
    \Cref{eta_asympotics}, we compute a higher order asymptotic expansion of the nonlinear term,
    \[\eta(u)(\eta''(u) + \chi A''(u)),\]
    appearing in \eqref{shape_defect_evolution_eqn}.
    \Cref{lower_bound_section} extends the lower bound for $u$
    found in \cite[Proposition 3.1]{shortproof}
    to the positive half line.
    Lastly, in \Cref{derivation}, we give the derivation of the shape defect function evolution equation from \eqref{shape_defect_evolution_eqn}.

	\subsection{Notation} \label{notation_section}
    From \eqref{weighted_nash_ineq}, we see that we are working in weighted $L^1$ spaces. For notational convenience,
    we set $L^1_g$ to be the $L^1$ space weighted by a function $g(x) > 0$. More precisely,
    \begin{equation}\label{weighted_space_def}
        L^1_g = \left\{\text{measurable } \psi: \norm{\psi}_{L^1_g} := \int |\psi| g dx < \infty\right\}.
    \end{equation}
    In addition, due to the adjoint argument outlined in \Cref{intro_adjoint_arg},
    we also make use of the dual of $L^1_g$:
    \begin{equation}\label{weighted_L_infty}
        L^\infty_{1/g} = \left\{\text{measurable } \psi : \norm{\psi}_{L^\infty_{1/g}} := \text{ess sup} \frac{|\psi|}{g} < \infty\right\}.
    \end{equation}

    Throughout this paper we make use of the following notation. The function $\tilde w$ is always
    the shape defect function shifted to the front location, as defined
    in \eqref{front_location_theorem}. Therefore, for $\chi = 1$,
    \begin{equation}
        \tilde w(t, x) = w\left(t, x + 2t - \frac{1}{2}\log t\right), \label{intro_pushmi_pullyu_tilda}
    \end{equation}
	while for $\chi \in [0, 1)$,
    \begin{equation}
        \tilde w(t, x) = w\left(t, x + 2t - \frac{3}{2}\log t\right). \label{intro_pulled_tilda}
    \end{equation}
    Similarly, $\tilde u$ is always the solution to \eqref{main_eq} shifted to the front
    location.
    
    Lastly, the symbol $C$ changes line-by-line, and represents a constant only depending on
    $\chi$, $A$, and $L_0$.

    \section{Convergence in the pushmi-pullyu case ($\chi = 1$)} \label{pushmi_pullyu_section}
    Following the intuition from \Cref{toy_model}, we consider
    \begin{equation}
        \bar w(t, x) = e^x \tilde w(t, x). \label{bar_w_rcl}
    \end{equation}
    The equations that $\bar w$ and $\tilde w$ satisfy are
    \begin{align}
        \tilde w_t &= \tilde w_{xx} + (1-A'(\tilde u))\tilde w - A'(\tilde u)\tilde w_x + \left(2 - \frac{1}{2t}\right)\tilde w_x \quad \text{and} \label{tilde_w_rcl_pde}\\
        \bar w_t &= \bar w_{xx} - A'(\tilde u)\bar w_x - \frac{1}{2t}(\bar w_x - \bar w).\label{bar_w_rcl_pde}
    \end{align}

    Up to shifting by a constant $x_0$, we can assume that
    $1-A'(\tilde u(t, x)) \leq - C < 0$ if $x \leq 0$. Hence, we would like to use \eqref{tilde_w_rcl_pde} since this term has the correct sign for $x \leq 0$.
    For $x > 0$, however, $1-A'(\tilde u)$ becomes positive, which does not help us. On the other hand,
    \eqref{bar_w_rcl_pde} does not have this growth term on the right, but we lose the $1-A'(\tilde u)$ term on the left.
    Therefore, our idea is to consider $h(x)\tilde w(t, x)$, where
    $h$ is constant on the left, $h(x) = e^x$ on the right, and $h$ connects the left and right in a smooth enough way.
    Let 
    \begin{equation}
        h(x) = \begin{cases}
        2e^{-1}  \quad &x \in (-\infty,-1],\\
        e^{1/x} + e^x  \quad &x \in (-1,0),\\
        e^x \quad &x \in [0,\infty).
    \end{cases} \label{interpolating_function}
    \end{equation}
    Note that $h \in C^1(\R)$ and $0 \leq h' \leq h$.
    We make one more simplifying transformation due to the $\frac{\bar w}{2t}$ term in \eqref{bar_w_rcl}. This term implies a growth of $\sqrt t$, so it makes sense to consider
    \begin{equation}
        \hat w(t, x) = \frac{1}{\sqrt t} h(x) \tilde w(t, x).
    \end{equation}

    We state two lemmas which, when combined, prove \Cref{shape_defect_theorem}.\eqref{i.pushmi-pullyu}.  The first gives the $L^2$-decay for $\hat w$.
    \begin{lemma}\label{forward_lemma}
    	Under the assumptions of \Cref{main_result_theorem},
        \begin{equation}
            \norm{\hat w(t)}_{L^2} \leq \frac{C}{t^{3/4}}. \label{pushmi_pullyu_L2_result}
        \end{equation}
    \end{lemma}
    The proof of \Cref{forward_lemma} can be found in \Cref{pushmi_pullyu_L2_decay_section}.
    The second lemma gives a weighted $L^2 \to L^\infty$-estimate.  For this, we define the weight
    \begin{equation}
        g(x) = \begin{cases}
        h(x) \quad &x \in (-\infty, 0)\\
        x+1 \quad &x \in [0, \infty)
    \end{cases} = \begin{cases}
        2e^{-1} \quad &x \in (-\infty, -1]\\
        e^x + e^{-1/x} \quad &x \in (-1, -0)\\
        x+1 \quad &x \in [0, \infty).
    \end{cases} \label{pushmi_pullyu_weight}
    \end{equation}
    The weight $g \in C^1(\R)$ and will be used in the weighted Nash inequality. It leads to the following $L^\infty$-bound.
    \begin{lemma} \label{adjoint_lemma}
        Let $0 \leq s < t$. Then, under the assumptions of \Cref{main_result_theorem},
        \begin{equation}
            \norm{\hat w(t)}_{L^\infty_{1/g}} \leq \frac{C\norm{\hat w(s)}_{L^2}}{(t-s)^{3/4}}. \label{pushmi_pullyu_Linfinity_result}
        \end{equation}
    \end{lemma}
    The proof of \Cref{adjoint_lemma} can be found in \Cref{pushmi_pullyu_adjoint_section}.
    Before proving \Cref{forward_lemma, adjoint_lemma}, we show how they yield \eqref{pushmi_pullyu_w_decay} in \Cref{shape_defect_theorem}.

            \begin{proof}[Proof of {\Cref{shape_defect_theorem}}.\eqref{i.pushmi-pullyu}]
                Using \Cref{adjoint_lemma} with $s = t/2$ yields
                \[\norm{\hat w(t)}_{L^\infty_{1/g}} \leq \frac{C}{(t/2)^{3/4}} \norm{\hat w(t/2)}_{L^2}.\]
                Next we apply \Cref{forward_lemma} to obtain
                \[\norm{\hat w(t, \cdot)}_{L^\infty_{1/g}} \leq \frac{C}{(t/2)^{3/4}} \frac{C}{(t/2)^{3/4}} = \frac{C}{t^{3/2}}.\]
                Recall that $\hat w(t, x) = \frac{1}{\sqrt t} h(x)\tilde w(t, x)$. Thus,
                \[\left|\frac{h(x) \tilde w(t, x)}{g(x)}\right| \leq \frac{C}{t}.\]
                Using the definitions of $g(x)$ and $h(x)$, \eqref{interpolating_function} and \eqref{pushmi_pullyu_weight} respectively, completes the proof of
                \Cref{shape_defect_theorem}.\eqref{i.pushmi-pullyu}.
            \end{proof}

    \subsection{$L^2$-estimates: proof of \Cref{forward_lemma}} \label{pushmi_pullyu_L2_decay_section}
    \begin{proof}
    In the upcoming proof, our goal is to manipulate $\int (\hat w^2)_t dx$ so that we may apply the classical Nash inequality. Then
    we estimate the $L^1$-norm and finally solve a differential inequality for $\int \hat w^2 dx$.

    We calculate how the $L^2$-norm changes in time:
    \begin{align*}
        \frac{d}{dt}\int \hat w^2 dx &= \frac{d}{dt}\left[ \frac{1}{t}\int (\tilde w h(x))^2 dx \right]\\
        &= -\frac{1}{t^2}\int (\tilde w h(x))^2 dx + \frac{2}{t}\int \tilde w h(x)^2\left[\tilde w_{xx} + (1-A'(\tilde u))\tilde w - A'(\tilde u)\tilde w_x + 2\tilde w_x - \frac{1}{2t}\tilde w_x\right]dx\\
        &=-\frac{1}{t}\int \hat w^2 dx + \frac{1}{t}\int 2\tilde w h(x)^2\left[\tilde w_{xx} + (1-A'(\tilde u))\tilde w - A'(\tilde u)\tilde w_x + 2\tilde w_x - \frac{1}{2t}\tilde w_x\right]dx\\
        &=: -\frac{1}{t}\int \hat w^2 dx + \frac{1}{t}I.
    \end{align*}
    Note that we have defined $I$ implicitly in the last equality.  We need to understand $I$ now.  To do so, we integrate by parts to find
    \begin{align*}
        I = -2&\int \tilde w_x^2 h(x)^2 dx - 4\int \tilde w \tilde w_x h(x) h'(x)dx +2 \int \tilde w^2(1-A'(\tilde u))h(x)^2dx\\
         &- \int(\tilde w^2)_xA'(\tilde u)h(x)^2dx + 2\int(\tilde w^2)_x h(x)^2dx - \frac{1}{2t}\int (\tilde w^2)_x h(x)^2dx.
    \end{align*}
    First note that \[(\tilde wh(x))_x^2 = (\tilde w_x h(x) + \tilde w h'(x))^2 = \tilde w_x^2 h(x)^2 + 2\tilde w_x \tilde w h(x) h'(x) + \tilde w^2 h'(x)^2.\]
    Therefore, we can write $I$ as
    \begin{align*}
        I = -2\int (\tilde w h(x))_x^2 dx + 2&\int \tilde w ^2h'(x)^2dx +2 \int \tilde w^2(1-A'(\tilde u))h(x)^2dx\\
         & +\int\tilde w^2\left[A''(\tilde u)\tilde u_x h(x)^2 + 2A'(\tilde u)h(x)h'(x)\right]dx\\
         &- 4\int\tilde w^2 h(x)h'(x)dx + \frac{1}{t}\int \tilde w^2 h(x)h'(x)dx.
    \end{align*}
    The previous expression for $I$ might look daunting, but consider the case when $x \leq -1$. Since $h(x)$ is constant in this region,
    all the terms with $h'(x)$ vanish and we are left with
    \begin{equation}\label{pushmi_pullyu_left}
        -2\int_{-\infty}^{-1} (\tilde w h(x))_x^2 dx +2 \int_{-\infty}^{-1} \tilde w^2(1-A'(\tilde u))h(x)^2dx +\int_{-\infty}^{-1}\tilde w^2A''(\tilde u)\tilde u_x h(x)^2 \leq -2\int_{-\infty}^{-1} (\tilde w h(x))_x^2 dx.
    \end{equation}
    In the previous inequality we used that $1-A'(\tilde u(t, x)) \leq 0$ for $x \leq 0$ and that $\tilde u_x \leq 0$.
    On the other hand, if $x \geq 0$, then $h(x) = e^x = h'(x)$. So on the right, $I$ becomes
    \begin{equation}\label{pushmi_pullyu_right}
        -2\int_0^\infty (\tilde w h(x))_x^2 dx  + \int_0^\infty \hat w^2 A''(\tilde u) \tilde u_x dx + \frac{1}{t}\int_0^\infty \hat w^2 dx \leq  -2\int_0^\infty (\tilde w h(x))_x^2 dx + \frac{1}{t}\int_0^\infty \hat w^2 dx.
    \end{equation}
    Therefore, we have nice cancellation for $x \notin (-1, 0)$. The tricky part is determining whether a similar cancellation occurs over $(-1, 0)$.

    Consider all but the first and last integrands in $I$ and set
    \[S:=2h'(x)^2 + 2(1-A'(\tilde u))h(x)^2 + A''(\tilde u)\tilde u_x h(x)^2 + 2A'(\tilde u)h(x)h'(x) - 4h(x)h'(x).\]
    Since $0 \leq h' \leq h$, we have
    \[
    2h'(x)^2 - 4h(x)h'(x) \leq -2 h(x)h'(x).
    \] 
    In addition $A''(\tilde u)u_x h(x)^2 \leq 0$, so we are left with
    \begin{equation}\label{rcl_integrand_ineq}
        \begin{split}
            S &\leq -2h(x)h'(x) + 2(1-A'(\tilde u))h(x)^2 + 2A'(\tilde u)h(x)h'(x)\\
                &= 2h(x)[-h'(x) + (1-A'(\tilde u))h(x) + A'(\tilde u)h'(x)]\\
                &=2h(x)[(1-A'(\tilde u))h(x) + h'(x)(A'(\tilde u)-1)]\\
                &= 2h(x)(h(x)-h'(x))[1-A'(\tilde u)].
        \end{split}
    \end{equation}
    Once again using that $1-A(\tilde u) \leq 0$ for $x \leq 0$ and $h' \leq h$, we find that $S \leq 0$ on $(-1, 0)$.
    We already saw from \eqref{pushmi_pullyu_left} and \eqref{pushmi_pullyu_right} that we have
    favorable bounds in the complement of $(0, 1)$, leaving us with
    \begin{align*}
        \frac{d}{dt} \int \hat w^2 dx &\leq -\frac{1}{t}\int \hat w^2 dx -2\int \hat w_x^2 dx + \frac{1}{t}\int_0^\infty \hat w^2 dx\\
        &\leq -2\int \hat w_x^2 dx.
    \end{align*}
    We have reached a point at which we can apply the classical Nash inequality to obtain
    \begin{equation}
        \frac{d}{dt}\int \hat w^2dx \leq -C\frac{\left(\int \hat w^2dx\right)^3}{\left(\int \hat w dx\right)^4}. \label{pushmi_pullyu_dissipation}
    \end{equation}

    The next step is to find a bound on $\norm{\hat w(t)}_{L^1}$. Then, we can solve the differential inequality for $\norm{\hat w(t)}_{L^2}^2$.
    We calculate the $L^1$-norm directly:
    \begin{equation}\label{pushmi_pullyu_direct_L1}
        \int \hat w dx \leq \frac{1}{\sqrt t}\int (1+e^x)\tilde wdx = \frac{1}{\sqrt t}\int \tilde w dx + \frac{1}{\sqrt t}\int e^x \tilde w dx.
    \end{equation}
    The first term can be easily bounded
    using the definition \eqref{shape_defect_function} of the shape defect function,
    integrating by parts,
    and recalling that $\tilde u(-\infty)=1$ and $\tilde u(+\infty) = 0$. Thus,
    \begin{equation}\label{pushmi_pullyu_L1_easy}
        \int \tilde wdx = \int \left[-\tilde u_x - (\tilde u-A(\tilde u))\right]dx \leq -\int \tilde u_x dx= 1.
    \end{equation}
    For the second term in \eqref{pushmi_pullyu_direct_L1}, since we have shifted to the front,
    \begin{equation}\label{pushmi_pullyu_L1_hard}
        \int e^x \tilde wdx = \int e^x(-\tilde u_x - (\tilde u-A(\tilde u)))dx = \int e^x A(\tilde u)dx \leq C.
    \end{equation}
    Therefore inserting the bounds \eqref{pushmi_pullyu_L1_easy} and \eqref{pushmi_pullyu_L1_hard} into \eqref{pushmi_pullyu_direct_L1}, we obtain
    \[\int \hat w dx \leq \frac{C}{\sqrt t}.\]
    Using this inequality in \eqref{pushmi_pullyu_dissipation}, we have
    \begin{equation}\label{pushmi_pullyu_can_solve}
        \frac{d}{dt}\int \hat w^2dx \leq -Ct^2\left(\int \hat w^2dx\right)^3.
    \end{equation}
    Solving \eqref{pushmi_pullyu_can_solve} for $\norm{\hat w(t)}_{L^2}$ yields
    \[\int \hat w^2 dx \leq \frac{C}{t^{3/2}},\]
    which completes the proof.
    \end{proof}

    \subsection{$L^2 \to L^\infty$-estimates: proof of \Cref{adjoint_lemma}} \label{pushmi_pullyu_adjoint_section}
            We now prove \Cref{adjoint_lemma}. The key step in the upcoming
            proof is how we make use of the weighted Nash inequality \eqref{weighted_nash_ineq}.

            \begin{proof}[Proof of {\Cref{adjoint_lemma}}]
                To find the correct $L^\infty$-decay, we employ
                the adjoint argument described in \Cref{intro_adjoint_arg}.
                First we compute the adjoint equation. Let $v$ be a test function, consider $\int v w_t dx$, and integrate by parts:
                \begin{align*}
                    \int v \hat w_t dx &= \int v h(x)\frac{1}{\sqrt{t}}\left[-\frac{\tilde w}{2t} + \tilde w_{xx} + (1-A'(\tilde u))\tilde w - A'(\tilde u)\tilde w_x + 2\tilde w_x - \frac{1}{2t}\tilde w_x\right]dx\\
                    &= \int \left[-\frac{vh}{2t} + (vh)_{xx} + (1-A'(\tilde u))vh + (A'(\tilde u) vh)_x -2(vh)_x + \frac{1}{2t}(vh)_x\right]\frac{\tilde w}{\sqrt{t}}dx\\
                    &= \int \frac{1}{h}\left[-\frac{vh}{2t} + (vh)_{xx} + (1-A'(\tilde u))vh + (A'(\tilde u) vh)_x -2(vh)_x + \frac{1}{2t}(vh)_x\right]\hat wdx.
                \end{align*}
                In the last equality we multiplied and divided by $h(x)$ and used that $\hat w(t, x) = \frac{h(x)\tilde w(t, x)}{\sqrt t}$.
                Therefore, the adjoint equation is given by
                \begin{equation}\label{push_pull_rcl_adjoint_pde}
                    v_t = -\frac{v}{2t} + \frac{(vh)_{xx}}{h} + (1-A'(\tilde u))v + \frac{(A'(\tilde u)v h)_x}{h}-\frac{(2vh)_x}{h} + \frac{1}{2t} \frac{(vh)_x}{h}.
                \end{equation}
                Our goal is to prove an
                $L^1_g \to L^2$ bound for \eqref{push_pull_rcl_adjoint_pde}. Similar to the proof of \Cref{forward_lemma}, we aim to solve
                a differential inequality involving $\norm{v}_{L^2}$. The difference is that we use a weighted Nash inequality
                with the weight defined in \eqref{pushmi_pullyu_weight}.

                Let $v$ be the solution to \eqref{push_pull_rcl_adjoint_pde} with initial data $v_0$.  We compute how the $L^2$-norm changes:
                \begin{align*}
                    \frac{d}{dt} \int v^2 dx &= 2 \int v v_tdx\\
                    &= 2\int v\left[-\frac{v}{2t} + \frac{(vh)_{xx}}{h} + (1-A'(\tilde u))v + \frac{(A'(\tilde u)v h)_x}{h}-\frac{(2vh)_x}{h} + \frac{1}{2t} \frac{(vh)_x}{h}\right]dx.
                \end{align*}
                After integrating by parts, we find
                \begin{alignat*}{2}
                    2&\int v \frac{(vh)_{xx}}{h}dx &&= -2\int v_x^2 dx + 2\int \frac{v^2 h'(x)^2}{h^2}dx,\\
                    2&\int \frac{v}{h}(A'(\tilde u) hv)_x dx &&= \int v^2 A''(\tilde u)\tilde u_x dx + 2\int v^2 A'(\tilde u)\frac{h'}{h}dx,\\
                    2\left(-2 + \frac{1}{2t}\right)&\int \frac{v}{h}(vh)_xdx &&= -2\left(-2 + \frac{1}{2t}\right)\int v^2 \frac{h'}{h}dx.
                \end{alignat*}
                Combining the previous three calculation, we arrive at
                \begin{equation}\label{pushmi_pullyu_adjoint_ibp}
                    \begin{split}
                        \frac{d}{dt} \int v^2 dx = -\frac{1}{t}\int v^2dx -2&\int v_x^2 dx + \frac{1}{t}\int v^2 \frac{h'}{h}dx\\
                    &+ \int v^2 \left[\frac{2h'(x)^2}{h^2} + 2(1-A'(\tilde u)) + A''(\tilde u)\tilde u_x + \frac{2A'(\tilde u)h'}{h} - \frac{4h'}{h}\right]dx.
                    \end{split}
                \end{equation}
                The integrand of the last integral in \eqref{pushmi_pullyu_adjoint_ibp} can be rewritten as
                \[\frac{v^2}{h(x)^2}\left[2h'(x)^2 + 2(1-A'(\tilde u))h(x)^2 + A''(\tilde u)\tilde u_x h(x)^2 + 2A'(\tilde u)h'(x)h(x) - 4h'(x)h(x)\right].\]
                Notice the expression inside the brackets is exactly \eqref{rcl_integrand_ineq}, which we showed is nonpositive. In addition,
                since $h' \leq h$, we have
                \begin{align*}
                    \frac{d}{dt}\int v^2 dx &\leq -2\int v_x^2 dx + 2\int_{-\infty}^{0} v^2(1-A'(\tilde u))dx.
                \end{align*}
                Since $\tilde u$ has is shifted to the front, there exists a $C > 0$, so that $1-A'(\tilde u(t, x)) \leq - C$ for all $x \leq 0$.
                \[\frac{d}{dt}\int v^2 dx \leq -2\int v_x^2 dx -C\int_{-\infty}^{0} v^2 dx.\]

                We now use the weighted Nash inequality with the weight given by \eqref{pushmi_pullyu_weight}. There exists $C_1, C_2 > 0$
                such that for any $L > 0$ both of the following hold:
                \begin{equation}
                    -\int v^2_x dx \leq {\begin{cases}
                    \displaystyle
                    \frac{-C_1}{L^2}\int_{0}^\infty v^2 dx + \frac{C_2}{L^5}\left(\int_{0}^\infty (x+1)v dx\right)^2\\
                    \displaystyle
                    \frac{-C_1}{L^2} \int v^2 dx + \frac{C_2}{L^3}\left(\int v dx\right)^2.
                \end{cases}} \label{weighted_nash_1}
                \end{equation}
                The first case in \eqref{weighted_nash_1} comes from the weighted Nash inequality with $\beta = 1$. The second case
                is really just the classical Nash inequality and
                can be obtained from \eqref{weighted_nash_ineq} by choosing $\beta = 0$.
                If $L \geq 1$, the first case gives us
                \begin{equation}
                    \begin{split}
                        \frac{d}{dt}\int v^2 dx &\leq - 2\int v^2_x dx - C\int_{-\infty}^{0} v^2 dx \\
                    &\leq \frac{-C_1}{L^2}\int_{0}^\infty v^2 dx + \frac{C_2}{L^5}\left(\int_{0}^\infty v(x+1)dx\right)^2 - C\int_{-\infty}^{0} v^2 dx\\
                    &\leq \frac{-C_1}{L^2}\int v^2 dx + \frac{C_2}{L^5}\left(\int vg(x)dx\right)^2.
                    \end{split}
                \end{equation}
                If $0 < L < 1$, the second case gives us
                \begin{equation}\label{pushmi_pullyu_adjoint_dissipation}
                    \begin{split}
                        \frac{d}{dt}\int v^2 dx &\leq -2\int v_x^2 dx \leq \frac{-C_1}{L^2}\int v^2 dx + \frac{C_2}{L^3}\left(\int vdx\right)^2\\
                        &\leq \frac{-C_1}{L^2}\int v^2 dx + \frac{C_2}{L^5}\left(\int vg(x)dx\right)^2.
                    \end{split}
                \end{equation}
                Therefore, we get the same dissipation inequality in either case. Inequality \eqref{pushmi_pullyu_adjoint_dissipation} holds true for any $L > 0$. 
                In particular, we can choose $L$ such that
                \[\frac{1}{L^2}\int v^2dx = \frac{B}{L^5}\left(\int vg(x)dx\right)^2,\]
                where $B > 0$ is to be chosen later.
                Solving for $L$, we have
                \begin{equation}
                    L = \frac{B^{1/3}\left(\int v g(x)dx\right)^{2/3}}{\left(\int v^2 dx\right)^{1/3}}. \label{pushmi_pullyu_chosen_L}
                \end{equation}
                Using \eqref{pushmi_pullyu_chosen_L} in \eqref{pushmi_pullyu_adjoint_dissipation}, we obtain
                \[\frac{d}{dt}\int v^2 dx \leq \frac{\left(\int v^2 dx\right)^{5/3}}{\left(\int v g(x)dx\right)^{4/3}}\left(\frac{-C_1}{B^{2/3}} + \frac{C_2}{B^{5/3}}\right).\]
                Finally take $B > \frac{C_2}{C_1}$ to find
                \begin{equation}
                    \frac{d}{dt}\int v^2 dx \leq -C \frac{\left(\int v^2 dx\right)^{5/3}}{\left(\int v g(x)dx\right)^{4/3}}. \label{push_pull_rcl_dissipation}
                \end{equation}

               To complete the proof, we need a precise estimate on
                $\int v g(x)dx$. We calculate
                \begin{equation}\label{pushmi_pullyu_adjoint_L1_ibp}
                    \begin{split}
                        \frac{d}{dt}\int v g(x)dx = -\frac{1}{2t}\int vg(x)dx + &\int \frac{(v h)_{xx}}{h} g dx + \int v(1-A'(\tilde u))g dx + \int \frac{(A'(\tilde u) hv)_x}{h} gdx\\
                                &-2 \int \frac{(vh)_x}{h} g dx + \frac{1}{2t}\int \frac{(vh)_x}{h}g dx.
                    \end{split}
                \end{equation}
                After integrating by parts, we find
                \begin{alignat*}{2}
                    &\int \frac{(v h)_{xx}}{h} g dx &&= \int v \left(g'' - \frac{h''}{h} g + 2\frac{h'(x)^2}{h^2}g - 2\frac{h'}{h}g'\right)dx,\\
                    &\int \frac{(A'(\tilde u) hv)_x}{h} gdx &&= -\int vA'(\tilde u) \left(g' -\frac{h'}{h}g\right)dx,\\
                    \left(-2 + \frac{1}{2t}\right)&\int \frac{(vh)_x}{h} g dx &&=  \left(2-\frac{1}{2t}\right) \int v \left(g' - \frac{h'}{h}g\right)dx.
                \end{alignat*}
                Recall that on $(-\infty, 0)$, $h(x) = g(x)$, while on $[0, \infty)$, $h(x) = e^x$ and $g(x) = x+1$. Therefore,
                \begin{alignat*}{3}
                    &\int \frac{(v h)_{xx}}{h} g dx &&= \int_0^\infty v \left(0 -(x+1) + 2(x+1) - 2\right)dx &&= \int_0^\infty v \left(x-1\right)dx,\\
                    &\int \frac{(A'(\tilde u) hv)_x}{h} gdx &&= -\int_0^\infty v A'(\tilde u) (1-(x+1)) dx &&= \int_0^\infty v A'(\tilde u) x dx,\\
                    \left(-2 + \frac{1}{2t}\right)&\int \frac{(vh)_x}{h} g dx &&= \left(2-\frac{1}{2t}\right) \int_0^\infty v\left(1 - (x+1)\right)dx &&= \left(-2+\frac{1}{2t}\right) \int_0^\infty v x dx.
                \end{alignat*}
                Using the previous calculations in \eqref{pushmi_pullyu_adjoint_L1_ibp} along with $1-A'(\tilde u) \leq 0$ for $x \leq 0$, we have
                \begin{align*}
                    \frac{d}{dt}\int v g(x)dx &\leq -\frac{1}{2t}\int v(x+1)dx - \int_0^\infty v(x+1)dx + \int_0^\infty v(1-A'(\tilde u))(x+1)dx\\
                    & \hspace*{4cm}+ \int_0^\infty v A'(\tilde u)x dx+ \frac{1}{2t}\int v xdx\\
                    &\leq -\int_0^\infty vA'(\tilde u)xdx \leq 0
                \end{align*}
                This implies that
                \begin{equation}
                    \int v(t, x) g(x)dx \leq \int v_0 g(x)dx. \label{adjoint_pushmi_pullyu_L1}
                \end{equation}

                Now we use \eqref{adjoint_pushmi_pullyu_L1} in \eqref{push_pull_rcl_dissipation}
                to find
                \[\frac{d}{dt}\int v^2 dx \leq -C \frac{\left(\int v^2 dx\right)^{5/3}}{\left(\int v_0 g(x)dx\right)^{4/3}}.\]
                Finally, we solve this differential inequality to obtain
                \begin{equation}
                    \int v^2 dx \leq \frac{C \norm{v_0}_{L^1_g}^2}{t^{3/2}}. \label{pushmi_pullyu_L2_bound}
                \end{equation}

                The estimate \eqref{pushmi_pullyu_L2_bound} means that the solution operator $\tilde S_t: L^1_g \to L^2$
                associated with the formal adjoint equation \eqref{push_pull_rcl_adjoint_pde}
                has bound
                \begin{equation*}
                    \norm{\tilde S_t}_{L^1_g \to L^2} \leq \frac{C}{t^{3/4}}.
                \end{equation*}
                Hence, the adjoint operator, $\tilde S_t^*: L^2 \to L^\infty_{1/g}$ also has the bound
                \begin{equation}
                    \norm{\tilde S_t^*}_{L^1_g \to L^2} \leq \frac{C}{t^{3/4}}. \label{pushmi_pullyu_adjoint_bound}
                \end{equation}
                Let $0 \leq s < t$. Since $\tilde S_t^*$  is a solution operator for $\hat w$, we can write
                \[\hat w(t, \cdot) = \tilde S_{t-s}^* \hat w(s, \cdot).\]
                By \eqref{pushmi_pullyu_adjoint_bound},
                \[\norm{\hat w(t)}_{L^\infty_{1/g}} = \norm{\tilde S_{t-s}^* \hat w(s)}_{L^\infty_{1/g}} \leq \frac{C\norm{\hat w(s)}_{L^2}}{(t-s)^{3/4}},\]
                which completes the proof.
            \end{proof}

    \section{Convergence in the pulled case ($0 \leq \chi < 1$)} \label{pulled_section}
    When $0 \leq \chi < 1$, the calculations become more tedious, but the main ideas are
    analogous to the $\chi = 1$ case. We first find the $L^2$-decay of the shape defect function and then
    use an adjoint argument to get a weighted $L^2 \to L^\infty_{1/g}$-bound on the shape defect function.

    The shape defect function satisfies the following PDE:
    \begin{equation}
        w_t = w_{xx} - \chi A'(u)w_x + (1- A'(u))w + 2w(w+\eta(u))(\eta''(u) + \chi A''(u)). \label{rcl_pulled_pde}
    \end{equation}
    Its derivation is given in \Cref{derivation}. Similar to the $\chi = 1$ case, we shift to the front
    at $m(t) = 2t - \frac{3}{2}\log t$, and scale by the same interpolating function \eqref{interpolating_function} as in \Cref{pushmi_pullyu_section}.
    We give its definition again for the reader's convenience:
    \begin{equation}
        h(x) = \begin{cases}
        2e^{-1}  \quad &x \in (-\infty,-1],\\
        e^{1/x} + e^x  \quad &x \in (-1,0),\\
        e^x \quad &x \in [0,\infty).
    \end{cases} \label{pulled_interpolating_function}
    \end{equation}
    Now we consider
    \begin{equation}
        \bar w(t, x) = h(x)\tilde w(t, x) = h(x) w(t, x + m(t)).
    \end{equation}
    In addition, because we have shifted to the $2t -\frac{3}{2}\log t$ moving frame, we let
    \begin{equation}
        \hat w (t, x) = \frac{1}{t^{3/2}} \bar w(t, x). \label{pulled_hat_w}
    \end{equation}
    Let us note that up to shifting by a constant $x_0$, we can assume $1- A'(\tilde u(t, x)) \leq -C < 0$
    for $x \leq 0$.

    We now give the analogues of \Cref{forward_lemma,adjoint_lemma} for the pulled case
    which, when combined, yield \Cref{shape_defect_theorem}.\eqref{i.pulled}.
    The first lemma gives the $L^2$-decay for $\hat w$, and its proof can be found in \Cref{pulled_forward_decay_subsection}.
    \begin{lemma}\label{pulled_rcl_L2_lemma}
        Suppose the initial data $u_0$ satisfies the conditions of \Cref{main_result_theorem}. Then
        \begin{equation}\label{pulled_L2_decay}
            \norm{\hat w(t)}_{L^2} \leq \frac{C}{t^{5/4}}.
        \end{equation}
    \end{lemma}

    For the $L^\infty$-estimate we define the weight
    \begin{equation}
        g(x) = \begin{cases}
        M h(x) \quad & x \in (-\infty, 0)\\
        x^2 + Mx + M \quad & x \in [0, \infty)
    \end{cases} = \begin{cases}
        2M e^{-1} \quad & x \in (-\infty, -1]\\
        M(e^x + e^{1/x}) \quad & x \in (-1, 0)\\
        x^2 + Mx + M \quad & x \in (0, \infty).
    \end{cases} \label{pulled_weight}
    \end{equation}
    The constant $M > 0$ is to be chosen later.
    Also, note that $g(x) \in C^1(\R)$ is quadratic for $x \geq 0$, which differs from the linear weight in the $\chi = 1$ case.
    We use \eqref{pulled_weight} in the weighted Nash inequality and leads to the following $L^\infty$-bound.
    \begin{lemma}\label{rcl_pulled_adjoint_lemma}
         Let $0 < s < t$. Then, under the assumptions of \Cref{main_result_theorem},
         \begin{equation}\label{rcl_pulled_adjoint_result}
            \norm{\hat w(t)}_{L^\infty_{1/g}} \leq \frac{C\norm{\hat w(s)}_{L^2}}{(t-s)^{5/4}}.
         \end{equation}
    \end{lemma}
    The proof of \Cref{rcl_pulled_adjoint_lemma} can be found in \Cref{pulled_adjoint_subsection}.
    With \Cref{pulled_rcl_L2_lemma} and \Cref{rcl_pulled_adjoint_lemma}, we can now
    prove \Cref{shape_defect_theorem}.\eqref{i.pulled}.
    \begin{proof}[Proof of {\Cref{shape_defect_theorem}}.\eqref{i.pulled}]
        We follow the proof given in \Cref{pushmi_pullyu_section} for \Cref{shape_defect_theorem}.\eqref{i.pushmi-pullyu},
        except we replace \Cref{forward_lemma,adjoint_lemma} with \Cref{pulled_rcl_L2_lemma, rcl_pulled_adjoint_lemma}.
    \end{proof}

    Before we begin the proof of the $L^2$-decay, we state one more lemma
    about the concavity of $\eta + \chi A$.
    \begin{lemma} \label{concave}
        For $0 \leq \chi < 1$,
        \begin{equation}\label{concave_result}
            \eta''(u) + \chi A''(u) \leq 0 \quad \text{ on } (0, 1).
        \end{equation}
    \end{lemma}
    This result is important because
    as compared the $\chi = 1$ case,
    the pulled case has an extra term:
    \begin{equation}\label{pulled_nonlinear_term}
        2w(w+\eta(u))(\eta''(u) + \chi A''(u)).
    \end{equation}
    Understanding \eqref{pulled_nonlinear_term} is imperative in finding the sharp convergence rate,
    so \Cref{concave} tells us at the very least \eqref{pulled_nonlinear_term} is nonpositive.
    \Cref{concave} is proven in \Cref{concave_proof}.

    \subsection{$L^2$-estimates: proof of \Cref{pulled_rcl_L2_lemma}} \label{pulled_forward_decay_subsection}
    \begin{proof}[Proof of {\Cref{pulled_rcl_L2_lemma}}]
        The overarching idea is similar to \Cref{pushmi_pullyu_section}:
        we aim to calculate $\int (\hat w^2)_t dx$ and apply the classical Nash inequality. There is
        more difficulty in calculating $\norm{\hat w}_{L^1}$, but it still mirrors
        the approach for $\chi = 1$.

    Using the PDE~\eqref{rcl_pulled_pde} for $w$, we have
    \begin{equation} \label{pulled_L2_calulation}
        \begin{split}
            \frac{d}{dt}\int \hat w^2 dx &= \frac{d}{dt}\left[\frac{1}{t^3}\int (\tilde w h(x))^2 dx \right]\\
            &= -\frac{3}{t^3}\int \hat w^2 dx + \frac{1}{t^3} \int 2\tilde w h(x)^2 (\tilde w_{xx} +(1-A'(\tilde u))\tilde w - \chi A'(\tilde u)\tilde w_x) dx\\
            & \hspace*{2cm}+  \frac{1}{t^3}\int 2\tilde w h(x)^2 \left[2\tilde w(\tilde w + \eta(\tilde u))(\eta'' + \chi A'')  + \left(2 - \frac{3}{2t}\right)\tilde w_x\right]dx\\
            &=: -\frac{3}{t}\int \hat w^2 dx + \frac{1}{t^3} I.
        \end{split}
    \end{equation}
    Note we have defined $I$ implicitly in the last equality. After integrating by parts $I$ becomes
    \begin{equation}\label{pulled_L2_ibp}
        \begin{split}
            I = -2&\int \hat w_x^2 dx + \int \tilde w^2 \left[2h'(x)^2 + 2(1-A'(\tilde u)) h(x)^2 + \chi A''(\tilde u)\tilde u_x h(x)^2\right]dx\\
            &+ \int \tilde w^2 \left[2 \chi A'(\tilde u)h(x)h'(x) + 4(\tilde w + \eta(\tilde u))(\eta'' + \chi A'')h(x)^2 - 4 h(x)h'(x) \right]dx\\
            &+ \frac{3}{t}\int \tilde w^2 h(x)h'(x) dx.
        \end{split}
    \end{equation}
    We aim to show that the sum of the middle two integrals in \eqref{pulled_L2_ibp} is nonpositive.
    First note that $\chi A''(\tilde u) \tilde u_x h(x)^2 \leq 0$ and by \Cref{concave},
    $\eta'' + \chi A'' \leq 0$.
    Therefore, we consider
    \begin{equation}
        S := 2h'(x)^2 + 2(1-A'(\tilde u))h(x)^2 + \chi A''(\tilde u)\tilde u_x h(x)^2 + 2\chi A'(\tilde u)h(x)h'(x) - 4 h(x)h'(x). \label{pulled_S_def}
    \end{equation}
    Recall that $0 \leq h' \leq h$, so \[2h'(x)^2 - 4h(x)h'(x) \leq 2h(x)h'(x) - 4h(x)h'(x) = -2h(x)h'(x).\]
    Hence,
    \begin{align*}
        S &\leq 2h(x) \left(-h'(x) + \chi A'(\tilde u)h'(x) + (1-A'(\tilde u))h(x)\right)\\
        &= 2h(x)(h'(x)(\chi A'(\tilde u) - 1) + (1-A'(\tilde u))h(x))\\
        &\leq 2h(x)(h'(x)(A'(\tilde u) - 1) + (1-A'(\tilde u))h(x))\\
        &= 2h(x)(1-A'(\tilde u))(h(x) - h'(x)).
    \end{align*}
    Since $h(x) = h'(x)$ for $x \geq 0$, we have $S \leq 0$ for $x \geq 0$. 
    On the other hand, $1-A'(\tilde u) \leq 0$ for $x \leq 0$ and $h' \leq h$.
    Thus, $S \leq 0$ for $x \leq 0$, implying $S \leq 0$ for all $x \in \R$.
    This means that
    \begin{equation}
        I \leq -2\int \hat w_x^2 dx + \frac{3}{t}\int \tilde w^2 h(x)h'(x) dx \leq -2\int \hat w_x^2 dx + \frac{3}{t}\int (\tilde w h(x))^2dx. \label{pulled_middle_integrals}
    \end{equation}
    In the last equality, we used that $h' \leq h$ again. Inserting \eqref{pulled_middle_integrals} in \eqref{pulled_L2_calulation},
    we obtain
    \[\frac{d}{dt}\int \hat w^2 dx \leq -\frac{3}{t}\int \hat w^2 dx - \frac{2}{t^3} \int (\tilde w h(x))_x^2dx + \frac{3}{t^4}\int (\tilde w h(x))^2dx = -2 \int \hat w_x^2 dx.\]
    Now we can apply the classical Nash inequality and find
    \begin{equation}\label{rcl_pulled_forward_dissipation}
        \frac{d}{dt}\int \hat w^2 dx \leq -C\frac{\left(\int \hat w^2dx\right)^{3}}{\left(\int \hat w dx\right)^4}.
    \end{equation}

    Below we show that
    \begin{equation}
        \int \hat w dx \leq \frac{C}{t}. \label{pulled_L1_estimate}
    \end{equation}
    Let us assume \eqref{pulled_L1_estimate} is true for now.  We show how to conclude the lemma under this assumption, and, afterwards, we verify \eqref{pulled_L1_estimate}.
    From this assumption, we deduce
    \[ \frac{d}{dt}\int \hat w^2dx\leq -Ct^4\left(\int \hat w^2 dx\right)^{3}. \]
    After solving this differential inequality for $\norm{\hat w(t)}_{L^2}$, one obtains
    the desired $L^2$-decay \eqref{pulled_L2_decay}. Therefore, all that remains is to show \eqref{pulled_L1_estimate}.

    To estimate $\norm{\hat w(t)}_{L^1}$,
    we first note that, by \eqref{prelim_estimate_result}, we have the preliminary bound
    \begin{equation}\label{pulled_L1_prelim}
    \tilde u(t, x) \leq CU_*(x) \leq C(x + 1) e^{-x}.
    \end{equation}
    Also, since $\hat w(t, x) = t^{-3/2}h(x) \tilde w(t, x) \leq t^{-3/2}(1+e^x)\tilde w(t, x)$,
    \begin{equation} \label{pulled_L1_first}
        \begin{split}
            t^{3/2}\int \hat w dx &\leq \int \tilde w (1+e^x)dx = \int (-\tilde u_x - \eta(\tilde u))(1+e^x)dx = 1 + \int (e^x(\tilde u - \eta(\tilde u)) - \eta(\tilde u))dx\\
            &\leq 1 + \int e^x(\tilde u - \eta(\tilde u))dx.
        \end{split}
    \end{equation}
    We integrated by parts in the second equality and used that $\tilde u(t, -\infty) = 1$.

    Now we estimate the integral in \eqref{pulled_L1_first}. For $a > 0$ to be determined below,
    \begin{equation}\label{pulled_L1_whole}
        \begin{split}
            \int e^x(\tilde u - \eta(\tilde u))dx &= \int_{-\infty}^{a} e^x(\tilde u - \eta(\tilde u)) dx + \int_{a}^\infty e^x(\tilde u - \eta(\tilde u))dx\\
            &\leq \int_{-\infty}^{a} e^x \tilde u dx + \int_{a}^\infty e^x(\tilde u - \eta(\tilde u))dx\\
            &\leq \int_{-\infty}^0 e^x dx + \int_{0}^{a} e^x \tilde u dx + \int_{a}^\infty e^x(\tilde u - \eta(\tilde u))dx\\
            &= 1 + \int_{0}^{a} e^x \tilde u dx + \int_{a}^\infty e^x(\tilde u - \eta(\tilde u))dx.
        \end{split}
    \end{equation}
    On $\left[0,a\right]$, we use \eqref{pulled_L1_prelim} to find
    \begin{equation}\label{pulled_L1_middle}
        \int_{0}^{a} e^x \tilde u dx \leq C\int_0^{a} (x+1) dx = C (a^2 + a).
    \end{equation}
    As a by-product of the proof of \Cref{front_location_theorem}, we prove an upper bound in \Cref{upperbound_prop}
    that defines constants $B$ and $T$, such that for $x \geq \frac{3}{2}\log(t+T) + \frac{3}{2}\log(t) + B$,
    \begin{equation} \label{pulled_shifted_u_bound}
         \tilde u(t, x) \leq C\left(x- \frac{3}{2}\log t\right)e^{-x}e^{-\frac{\left(x - \frac{3}{2}\log t - B\right)^2}{4(t+T)}}.
    \end{equation}
    Due to the asymptotics \cite[Lemma 3.4]{an2023locationdeterminesconvergencerate} for $\eta$, we have that for $\tilde u$ small,
    \[\tilde u - \eta(\tilde u) \leq \left|\frac{C \tilde u}{\log \tilde u}\right| = \frac{C \tilde u}{\log u^{-1}}.\]
    Using that $\frac{\tilde u}{\log \tilde u^{-1}}$ is an increasing function of $\tilde u$ for $0 < \tilde u < 1$, along with \eqref{pulled_shifted_u_bound}, we find that
    \begin{equation}\label{pulled_L1_upperbound}
        \tilde u - \eta(\tilde u) \leq \frac{-Cxe^{-x}e^{-\frac{(x - \frac{3}{2}\log t-B)^2}{4(t+T)}}}{\log \left(Cxe^{-x}e^{-\frac{(x - \frac{3}{2}\log t-B)^2}{4(t+T)}}\right)} = \frac{-Cxe^{-x}e^{-\frac{(x - \frac{3}{2}\log t-B)^2}{4(t+T)}}}{-x - \frac{(x-\frac{3}{2}\log t-B)^2}{4(t+T)} + \log(Cx)} \leq Ce^{-x}e^{-\frac{(x - \frac{3}{2}\log t-B)^2}{4(t+T)}}.
    \end{equation}
    The consequence of this calculation is that $\tilde u - \eta(\tilde u)$ ``removes'' the linear prefactor
    from $\tilde u$. If we had just used the bound $\tilde u - \eta(\tilde u) \leq \tilde u$, we would
    not obtain the correct $L^1$-estimate.

    Setting $a = \frac{3}{2}\log(t+T) + \frac{3}{2}\log(t) + B$ and
    plugging ~\eqref{pulled_L1_upperbound} into the integral over $[a, \infty)$ in ~\eqref{pulled_L1_whole} yields
    \begin{equation}\label{pulled_L1_right}
        \int_{a}^\infty e^x(\tilde u - \eta(\tilde u))dx \leq \int_{\frac{3}{2}\log(t+T) + \frac{3}{2}\log(t) + B}^\infty Ce^{-\frac{(x - \frac{3}{2}\log t-B)^2}{4t}}dx \leq C\int_0^\infty e^{-\frac{x^2}{4t}}dx \leq C \sqrt t.
    \end{equation}
    In addition, substituting $a$ in \eqref{pulled_L1_middle} gives us
    \begin{equation}\label{pulled_L1_middle_complete}
         \int_{0}^{a} e^x \tilde u dx \leq C\log(t)^2.
    \end{equation}
    Using  \eqref{pulled_L1_right} and \eqref{pulled_L1_middle_complete} in \eqref{pulled_L1_whole} and dividing by $t^{3/2}$, we obtain the desired estimate:
    \[\int \hat w dx \leq \frac{C}{t}.\]
    The proof is complete.
    \end{proof}

    \subsection{$L^2 \to L^\infty$ estimates: proof of \Cref{rcl_pulled_adjoint_lemma}}\label{pulled_adjoint_subsection}
    Before proving \Cref{rcl_pulled_adjoint_lemma}, we state a technical lemma which
    gives estimates on $2\eta(\tilde u)(\eta''(\tilde u) + \chi A''(\tilde u))$.
    \begin{lemma} \label{concave_technical_lemma}
        Let $u$ be a solution to \eqref{main_eq} with the assumption \eqref{initial_data_assumption}.
        There exists an $N > 0$ so that for all $x \geq N$,
        \begin{equation}
            2\eta(\tilde u)(\eta''(\tilde u) + \chi A''(\tilde u)) \leq \frac{-2}{x^2} + \frac{1}{xt} + \frac{C}{x^3}. \label{pulled_nonlinear_asymptotics}
        \end{equation}
    \end{lemma}
    The proof of \Cref{concave_technical_lemma} is given in \Cref{concave_proof}.

    \begin{proof}[Proof of {\Cref{rcl_pulled_adjoint_lemma}}]
    Similar to when $\chi = 1$, we employ an adjoint argument to find a $L^2_g \to L^\infty$ bound.
    The adjoint equation for $\hat w(t, x) = \frac{1}{t^{3/2}}h(x) \tilde w(t, x)$ is given by
    \begin{equation}\label{rcl_pulled_adjoint_eqn}
        \begin{split}
            v_t = \frac{(hv)_{xx}}{h} + &\frac{\chi (A'(\tilde u)hv)_x}{h} +(1-A'(\tilde u))v\\
        &+ 2v(\tilde w + \eta(\tilde u))(\eta''(\tilde u) + \chi A''(\tilde u)) - \frac{2(hv)_x}{h} + \frac{3(hv)_x}{2t} - \frac{3}{2t}v.
        \end{split}
    \end{equation}

    We start by calculating how the $\norm{v}_{L^2}^2$ changes in time:
    \begin{equation}\label{rcl_pulled_adjoint_dissipation}
        \begin{split}
            \frac{d}{dt} \int v^2 dx = -2&\int v_x^2 dx + \frac{3}{t}\int v^2 \left(\frac{h'}{h} - 1\right)dx + 4\int v^2 (\tilde w + \eta(\tilde u))(\eta''(\tilde u) + \chi A''(\tilde u))dx\\
        &+ \int v^2 \left[\frac{2h'(x)^2}{h(x)^2} + 2(1-A'(\tilde u)) + A''(\tilde u)\tilde u_x + \frac{2\chi A'(\tilde u)h'}{h} - \frac{4h'(x)}{h(x)}\right]dx.
        \end{split}
    \end{equation}
    We aim to show the last integral in \eqref{rcl_pulled_adjoint_dissipation} is negative. Define
    \[S := \frac{2h'(x)^2}{h^2} + 2(1-A'(\tilde u)) + A''(\tilde u)\tilde u_x + \frac{2\chi A'(\tilde u)h'}{h} - \frac{4h'(x)}{h(x)},\]
    which can be rewritten as
    \[
    	S = \frac{2}{h(x)^2}(h'(x)^2 + (1-A'(\tilde u))h(x)^2 + \chi A'(\tilde u)h'(x)h(x) - 4h'(x)h(x)).
    	\]
    By \eqref{pulled_S_def}, we have that 
    \[
    S \leq \frac{2}{h(x)}(1-A'(\tilde u))(h(x) - h'(x)) \leq 0.
    \]
    Just as in \Cref{pushmi_pullyu_adjoint_section}, we find
    $S \leq 0$ on $\R$ and, specifically for $x \leq 0$,
    \[
    S \leq 2(1-A'(\tilde u)) \leq -C.
    \]
    Now we use our bound on $S$ in \eqref{rcl_pulled_adjoint_dissipation}, along with the fact that $\eta'' + \chi A'' \leq 0$, to find
    \[\frac{d}{dt} \int v^2 dx \leq -2\int v_x^2 dx + \frac{3}{t}\int v^2 \left(\frac{h'(x)}{h(x)} - 1\right)dx - C \int_{-\infty}^{0} v^2 dx.\]
    Moreover, recall that $h' \leq h$, so $\frac{h'(x)}{h(x)} - 1 \leq  0$. Therefore, 
    \[\frac{d}{dt} \int v^2 dx \leq -2\int v_x^2 dx - C \int_{-\infty}^{0} v^2 dx.\]
    By the weighted Nash inequality, we obtain
    \begin{equation}
        \frac{d}{dt}\int v^2dx \leq -C\frac{\left(\int v^2 dx\right)^{7/5}}{\left(\int v g(x)dx\right)^{4/5}} \label{rcl_pulled_dissipation}
    \end{equation}
    using the exact same process as the $\chi = 1$ case, except replacing $\beta = 1$ with $\beta = 2$. 

    For the moment, let us assume that $\norm{v(t)}_{L^1_g}$ is decreasing in time.  We show how to conclude the lemma under this assumption, and, afterwards, we verify that $\norm{v(t)}_{L^1_g}$ is, in fact, decreasing.  From this assumption, we deduce
    \begin{equation}
        \frac{d}{dt}\int v^2dx \leq -C\frac{\left(\int v^2 dx\right)^{7/5}}{\left(\int v_0 g(x)dx\right)^{4/5}}. \nonumber
    \end{equation}
    Solving this differential inequality, we arrive at the desired $L^1_{g} \to L^2$ bound:
    \begin{equation}
        \norm{v}_{L^2} \leq \frac{C\norm{v_0}_{L^1_{\bar g}}}{t^{5/4}}.
    \end{equation}
    Following the adjoint argument from \Cref{adjoint_lemma}, we obtain for $0 \leq s < t$
    \[\norm{\hat w(t)}_{L^\infty_{1/g}} \leq \frac{C \norm{\hat w(s)}_{L^2}}{(t-s)^{5/4}}.\]
    Therefore, all that remains is to show that $\norm{v(t)}_{L^1_g}$ is decreasing, which we do now.

    To show that $\int v(t,x) g(x) dx$ is decreasing in time, we take its time derivative and show that it is nonpositive.
    Using \eqref{rcl_pulled_adjoint_eqn}, we see that
    \begin{align*}
        \frac{d}{dt}\int v g(x)dx = &\int \frac{(v h)_{xx}}{h} g dx + \int v (1-A'(\tilde u)) g(x)dx  + \chi\int \frac{(A'(\tilde u) hv)_x}{h} gdx\\
        &+ 2\int v(\tilde w + \eta(\tilde u))(\eta'' + \chi A'')g(x)dx -2 \int \frac{(vh)_x}{h} g dx + \frac{3}{2t}\int \frac{(vh)_x}{h}g dx - \frac{3}{2t}\int v g(x)dx.
    \end{align*}
    After integrating by parts, we find
    \begin{alignat*}{3}
        &\int \frac{(v h)_{xx}}{h} g dx &&= \int v \left(g'' - \frac{h''}{h} g + 2\frac{h'(x)^2}{h^2}g - 2\frac{h'}{h}g'\right)dx,\\
        \chi&\int \frac{(A'(\tilde u) hv)_x}{h} gdx &&= -\chi\int v A'(\tilde u) g' dx + \chi\int v A'(\tilde u)\frac{h'}{h}gdx,\\
        \left(-2 + \frac{3}{2t}\right) &\int \frac{(vh)_x}{h} g dx &&=  2 \int v g'dx - 2 \int v \frac{h'}{h}g dx -\frac{3}{2t} \int v g'dx + \frac{3}{2t} \int_0^\infty v \frac{h'}{h}g dx.
    \end{alignat*}
    Note that $g(x) = Mh(x)$ for $x \leq 0$. Using this fact on the three above integrals, implies
    \begin{alignat*}{3}
        &\int \frac{(v h)_{xx}}{h} g dx &&= \int_0^\infty v \left(g'' - \frac{h''}{h} g + 2\frac{h'(x)^2}{h^2}g - 2\frac{h'}{h}g'\right)dx,\\
        \chi&\int \frac{(A'(\tilde u) hv)_x}{h} gdx &&= -\chi\int_0^\infty v A'(\tilde u) g' dx + \chi\int_0^\infty v A'(\tilde u)\frac{h'}{h}gdx,\\
        \left(-2 + \frac{3}{2t}\right) &\int \frac{(vh)_x}{h} g dx &&=  2 \int_0^\infty v g'dx - 2 \int_0^\infty v \frac{h'}{h}g dx -\frac{3}{2t} \int_0^\infty v g'dx + \frac{3}{2t} \int_0^\infty v \frac{h'}{h}g dx.
    \end{alignat*}

    Hence,
    \begin{equation} \label{pulled_term_by_term}
        \begin{split}
            \frac{d}{dt}\int v g(x)dx = \int_0^\infty &vg''(x)dx + \int v(1-A'(\tilde u)) g(x)dx\\ 
            &+2\int v(\tilde w + \eta(\tilde u))(\eta'' + \chi A'')g(x)dx\\
            &+ \int_0^\infty v\left(-\frac{h''(x)}{h(x)} + \frac{2h'(x)^2 }{h(x)^2} - 2\frac{h'(x)}{h(x)} + \chi A'(\tilde u)\frac{h'(x)}{h(x)}\right) g(x) dx\\
            &+ \int_0^\infty v \left(2- \frac{2h'(x)}{h(x)} - \chi A'(\tilde u)\right)g'(x)dx - \frac{3}{2t}\int_0^\infty vg'(x)dx\\
            &+ \frac{3}{2t}\int_0^\infty v \left(-1 + \frac{h'(x)}{h(x)}\right)g(x)dx.
        \end{split}
    \end{equation}
    We now estimate term-by-term each integral in \eqref{pulled_term_by_term}.
    First, recall that $1-A'(\tilde u) \leq 0$ for $x \leq 0$, so that
    \begin{equation}\label{pulled_L1_intermediate1}
        \int v(1-A'(\tilde u)) g(x)dx \leq \int_0^\infty v(1-A'(\tilde u)) g(x)dx.
    \end{equation}
    In addition, for $x \geq 0$, $h(x) = e^x$. Therefore, for $x \geq 0$,
    \begin{equation}\label{pulled_L1_intermediate2}
        -\frac{h''(x)}{h(x)} + \frac{2h'(x)^2 }{h(x)^2} - 2\frac{h'(x)}{h(x)} + \chi A'(\tilde u)\frac{h'(x)}{h(x)} = -1 + \chi A'(\tilde u)
    \end{equation}
    and
    \begin{equation}\label{pulled_L1_intermediate3}
        2- \frac{2h'(x)}{h(x)} - \chi A'(\tilde u) = - \chi A'(\tilde u).
    \end{equation}
    Lastly, using \eqref{pulled_L1_intermediate1}, \eqref{pulled_L1_intermediate2}, and \eqref{pulled_L1_intermediate3},
    along with $g''(x) = 2$ for $x \geq 0$, 
    we can simplify $\frac{d}{dt} \int vg dx$ to
    \begin{equation}\label{pulled_L1_time_deriv}
        \begin{split}
            \frac{d}{dt}\int v g(x)dx \leq 2\int_0^\infty &vdx  + 2\int_0^\infty v(\tilde w + \eta(\tilde u))(\eta'' + \chi A'')g(x)dx\\
         &-\int_0^\infty vA'(\tilde u)((1 - \chi)g(x) + \chi g'(x))dx -\frac{3}{2t}\int_0^\infty v g'(x)dx.
        \end{split}
    \end{equation}
    At this point, all the terms on the right hand side of \eqref{pulled_L1_time_deriv} are nonpositive except
    for the first integral. Now we need to use the negative nonlinear term, $\eta(\tilde u)(\eta'' + \chi A'')$.

    To proceed, we separate into the ``$x$ large'' regime $[\bar N,\infty)$ and the remaining regime $[0,\bar N)$, where $\bar N$ is chosen below (see~\eqref{e.bar N}).  
    In the former regime, when $x$ is large, we aim to use \Cref{concave_technical_lemma},
    which allows the second and fourth integrals on the right hand side of~\eqref{pulled_L1_time_deriv} to cancel the first integral.
    In the latter regime, on which $x \sim O(1)$, we adjust $M$ in the weight $g$~\eqref{pulled_weight} in order to allow the third integral on the right hand side of~\eqref{pulled_L1_time_deriv} to cancel the first integral.

    We fix a preliminary constant $N$ as in \Cref{concave_technical_lemma}. Then, for $x \geq N$, \Cref{concave_technical_lemma} yields
    \begin{align}
        2\eta(\tilde u)(\eta''(\tilde u) + \chi A''(\tilde u)) g(x) &\leq \left(\frac{-2}{x^2}+\frac{1}{xt} + \frac{\bar C}{x^2}\right)\left(x^2 + Mx + M\right) \nonumber\\
        &= -2 -\frac{2M}{x} -\frac{2M}{x^2} + \frac{x}{t}+ \frac{M}{t} + \frac{M}{xt} + \frac{\bar C}{x} + \frac{\bar CM}{x^2} + \frac{\bar CM}{x^3}.  \label{nonlinearity_bound}
    \end{align}
    Consider the last integral in \eqref{pulled_L1_time_deriv}.  Observe that the weight in this integral satisfies
    \[-\frac{3}{2t} g'(x) = \frac{-3(2x + M)}{2t} \leq \frac{-3x}{t} -\frac{3M}{2t}.\]
    This compensates for all the terms with a $1/t$ in \eqref{nonlinearity_bound}, so that
    \[2\eta(\tilde u)(\eta''(\tilde u) + \chi A''(\tilde u)) g(x) -\frac{3}{2t} g'(x) \leq -2 -\frac{2M}{x} -\frac{2M}{x^2} + \frac{\bar C}{x} + \frac{\bar CM}{x^2} + \frac{\bar CM}{x^3}.\]
    Now we can choose $M > \bar C$ to obtain
    \[2\eta(\tilde u)(\eta''(\tilde u) + \chi A''(\tilde u)) g(x) -\frac{3}{2t} g'(x) \leq -2 - \frac{M}{x} -\frac{2M}{x^2} + \frac{\bar CM}{x^2} + \frac{\bar CM}{x^3}.\]
    Recall that $M$ is arises in~\eqref{pulled_weight}.  
    Observe that the last two terms above have a bad sign.  They can, however, be absorbed into the first two terms when $x$ is large.  This motivates the following definition: let
    \begin{equation}\label{e.bar N}
    		\bar N := \max\{\bar C, N\}.
    \end{equation}
    Thus,  when $x > \bar N$,
    \begin{equation}\label{pulled_nonlinear_weighted_inequality}
        2\eta(\tilde u)(\eta''(\tilde u) + \chi A''(\tilde u)) g(x) -\frac{3}{2t} g'(x) \leq -2.
    \end{equation}
    Using \eqref{pulled_nonlinear_weighted_inequality} in \eqref{pulled_L1_time_deriv}, we obtain
    \begin{align*}
        \frac{d}{dt}\int v g(x)dx
        &\leq 2\int_0^{\bar N} v dx  - \int_0^{\bar N} vA'(\tilde u)((1-\chi)g(x) + \chi g'(x))dx\\
        &\leq 2\int_0^{\bar N} v dx  - \int_0^{\bar N} vA'(\tilde u)((1-\chi)(x^2 + Mx +M) + \chi (2x + M))dx.\\
        &\leq 2\int_0^{\bar N} v dx  - M\int_0^{\bar N} vA'(\tilde u)dx.
    \end{align*}
    Further increasing $M$, if necessary, so that 
    \[M > \max\left\{\bar C, \frac{2}{A'(\tilde u(t, \bar N))}\right\},\]
    we deduce the desired inequality
    \begin{align*}
        \frac{d}{dt}\int vg(x) &\leq 0.
    \end{align*}
    This concludes the proof that $\norm{v(t)}_{L^1_g}$ is decreasing.
    \end{proof}

    In the upcoming subsection, we prove \Cref{concave,concave_technical_lemma}.
    \subsection{Useful properties of $\eta'' + \chi A''$} \label{concave_proof}

    First we give a lemma to show that $\eta + \chi A$ has enough regularity.
    \begin{lemma}\label{eta_regularity_lemma}
        Let $\eta$ be defined as in \eqref{profile_function}. Then
        \begin{equation}
            \eta + \chi A \in C^3(0, 1). \label{eta_regularity}
        \end{equation}
    \end{lemma}
    \begin{proof}
        Recall that the traveling wave profile function is given implicitly as
        \begin{equation}
            \eta(U_*) = -U_*'. \label{rcl_twpf}
        \end{equation}
        The traveling wave is a strong solution, so $U_*' \in C^1(0, 1)$, implying that $\eta \in C(0, 1)$.
        Therefore, using \eqref{rcl_twpf} with the differential equation \eqref{intro_TW_eqn} that $U_*$ satisfies, we find
        \[2\eta(U_*) = \chi A'(U_*)\eta(U_*) + \eta'(U_*)\eta(U_*) + U_* - A(U_*).\]
        Hence, we can write
        \[U_* - A(U_*) = \eta(U_*)(2-\eta'(U_*)) - \chi A'(U_*)\eta(U_*).\]
        Since $U_*$ is monotonically decreasing and continuous, with $U_*(-\infty) = 1$ and $U_*(+\infty) = 0$,
        for all $u \in [0, 1]$, we have
        \begin{equation}
            u-A(u) = \eta(u)(2-\eta'(u)) - \chi A'(u)\eta(u). \label{rcl_pulled_identity}
        \end{equation}
        Dividing by $\eta$ in \eqref{rcl_pulled_identity} and rearranging yields
        \begin{equation}
            \eta'(u) + \chi A'(u) = 2- \frac{u - A(u)}{\eta(u)}. \label{rcl_rearranged_pulled_identity}
        \end{equation}
        Recalling the assumptions \eqref{A_assumptions} on $A$, the fact that $\eta(u) \in (0, 1)$,
        and that $\eta \in C^1(0, 1)$, it follows that
        \[\eta' + \chi A' \in C^1(0, 1).\]
        In particular, $\eta' \in C^1(0, 1)$.
        We can now differentiate both sides of \eqref{rcl_rearranged_pulled_identity} to obtain
        \begin{equation}
            \eta''(u) + \chi A''(u) = -\frac{1- A'(u)}{\eta(u)} + \frac{(u-A(u))\eta'(u)}{\eta(u)^2}. \label{second_deriv_identity}
        \end{equation}
        Once again, by the assumptions \eqref{A_assumptions} on $A$, the first summand in
        \eqref{second_deriv_identity} is in $C^1(0, 1)$. The second summand is also in $C^1(0, 1)$,
        since $\eta' \in C^1(0, 1)$. Therefore, $\eta'' + \chi A'' \in C^1(0, 1)$, which
        completes the proof.
    \end{proof}

    Now we are ready to prove the concavity of $\eta + \chi A$.
    \begin{proof}[Proof of {\Cref{concave}}]
        First we show that for $u$ small enough, the desired inequality holds.
        To do this, use \eqref{second_deriv_identity} along with the identity
        \[\frac{u-A(u)}{\eta(u)} = 2-\eta'(u) - \chi A'(u),\]
        to find that
        \begin{align*}
        \eta''(u) + \chi A''(u) = \frac{A'(u)-1}{\eta(u)} - \frac{\eta'(u)(\chi A'(u) + \eta'(u)-2)}{\eta(u)}.
        \end{align*}
    Multiplying by $\eta$ yields
    \begin{equation}\label{eta_reduction_identity}
        \begin{split}
            \eta(u)(\eta''(u) + \chi A''(u)) &= A'(u) - 1 - \chi \eta'(u)A'(u)-\eta'(u)^2 + 2\eta'(u)\\
                &= -(1-\eta'(u))^2 + A'(u)(1-\chi \eta'(u)).
        \end{split}
    \end{equation}
    When $u$ is small, $\eta'(u) \sim 1 + \frac{1}{\log u}$ by \cite[Lemma 3.4]{an2023locationdeterminesconvergencerate}, while $A'(u) \leq Cu$.
    As a result,
    \begin{equation}\label{small_u_concave}
        \eta(u)(\eta''(u) + \chi A''(u)) \leq \frac{-1}{(\log u)^2} + Cu < 0.
    \end{equation}
    Since $\eta(u) > 0$ for $u \in (0, 1)$, the concavity holds for $u$ small enough.

    Our goal is to show that $\eta''(u) + \chi A''(u)$ remains nonpositive for $u \leq 1$.
    We argue by way of contradiction.  Suppose that $\eta'' + \chi A''$ does eventually becomes positive.
    Therefore, there exists
    \begin{equation}\label{crossing_point}
        \bar u = \inf \{u \in (0, 1) : \eta(u)^2 (\eta''(u) + \chi A''(u)) > 0\}.
    \end{equation}
    Note that $\bar u > 0$ by \eqref{small_u_concave}, and at $\bar u$ we have
    \begin{equation}\label{crossing_properties}
       \eta''(\bar u) + \chi A''(\bar u) = 0 \quad \text{ and } \quad [\eta^2 (\eta'' + \chi A'')]'(\bar u) \geq 0. 
    \end{equation}
    Using \eqref{second_deriv_identity}, we have the identity
    \begin{equation} \label{useful_second_deriv_identity}
        [\eta^2 (\eta'' + \chi A'')]'(u) = \left(\eta(u)(A'(u)-1) + \eta'(u)(u-A(u))\right)'= \eta''(u)(u-A(u)) + A''(u)\eta(u).
    \end{equation}
    Evaluating \eqref{useful_second_deriv_identity} at $\bar u$ and keeping in mind \eqref{crossing_properties}, we find
    \begin{equation}\label{not_strict_ineq}
        0 \leq [\eta^2(\eta'' + \chi A'')]'(\bar u) = \eta''(\bar u)(\bar u-A(\bar u)) + A''(\bar u)\eta(\bar u) = A''(\bar u)(\eta(\bar u) - \chi(\bar u- A(\bar u))).
    \end{equation}

    The assumption \eqref{A_assumptions} that $A(u)/u$ is increasing and convex implies that $A''(\bar u) > 0$.
    Therefore, the previous inequality \eqref{not_strict_ineq} implies $\eta(\bar u) - \chi (\bar u - A(\bar u)) \geq 0$. In fact,
    we can show strict inequality. 
    If not, then $\eta(\bar u) = \chi(\bar u - A(\bar u))$.
    Multiplying \eqref{second_deriv_identity} by $\eta^2$,
    we obtain
        \begin{equation}\label{strict_ineq_eq}
            \begin{split}
                0 &= \eta(\bar u)^2(\eta''(\bar u) + \chi A''(\bar u)) = \eta(\bar u) (A'(\bar u)-1) + \eta'(\bar u)(\bar u - A(\bar u))\\
                &= \chi (\bar u - A(\bar u))(A'(\bar u) - 1) + \eta'(\bar u)(\bar u - A(\bar u))\\
                &= (\bar u - A(\bar u))(\chi(A'(\bar u) - 1) +  \eta'(\bar u)).
            \end{split}
        \end{equation}
        Since $\bar u - A(\bar u) > 0$, \eqref{rcl_rearranged_pulled_identity} yields
        \[0 = -\chi + \chi A'(\bar u) + \eta'(\bar u) = \frac{A(\bar u)-\bar u}{\eta(\bar u)} +2 - \chi = -\frac{1}{\chi} + 2 - \chi.\]
        Solving for $\chi$ gives us a single solution $\chi = 1$, but we had assumed $\chi < 1$.
        This shows at $\bar u$
        \begin{equation}\label{cross_strict}
            \eta(\bar u) - \chi (\bar u - A(\bar u)) > 0.
        \end{equation}
        We are now able to complete the contradiction argument.

    With \eqref{cross_strict} and the definition \eqref{crossing_point} of $\bar u$,
    we can find $\tilde u \in (\bar u , 1)$ such that
    \begin{equation}\label{pos_properties}
        \eta''(\tilde u) + \chi A''(\tilde u) > 0, \quad [\eta^2(\eta'' + \chi A'')]'(\tilde u) > 0, \quad \text{ and } \quad \eta(\tilde u) - \chi(\tilde u - A(\tilde u)) > 0.
    \end{equation}
    The first two properties of \eqref{pos_properties} and that
    \[\lim_{u \to 1^-} \eta^2(\eta'' + \chi A'') = 0,\]
    imply that $\eta^2(\eta'' + \chi A'')$ must increase to a positive local maximum;
    that is, there exists a point $v \in (\tilde u, 1)$ such that
    \begin{equation}\label{local_max_properties}
        [\eta^2(\eta'' + \chi A'')]'(v) = 0 \quad \text{ and } \quad  \eta'' + \chi A'' > 0 \text{ on } [\tilde u, v].
    \end{equation}
    Using \eqref{useful_second_deriv_identity} in tandem with \eqref{local_max_properties},
    \begin{equation} \label{contradiction_ineq}
        \begin{split}
            0 = [\eta^2(\eta'' + \chi A'')]'(v) &= \eta''(v)(v-A(v)) + A''(v)\eta(v) \geq -\chi A''(v) (v-A(v)) + A''(v) \eta(v)\\
        &= A''(v)(\eta(v)-\chi(v-A(v))).
        \end{split}
    \end{equation}
    Since $A''(v) \geq 0$, \eqref{contradiction_ineq} implies
    \begin{equation}\label{useful_contradiction_ineq}
        \eta(v)-\chi(v-A(v)) \leq 0.
    \end{equation}
    On the other hand, since $\eta(\tilde u) - \chi(\tilde u-A(\tilde u)) > 0$, rearranging \eqref{rcl_pulled_identity} gives
        \begin{equation}\label{contradition_useful_1}
            1-\chi(2-\eta'(\tilde u)-\chi A'(\tilde u)) = \frac{\eta(\tilde u) - \chi(\tilde u-A(\tilde u))}{\eta(\tilde u)} > 0.
        \end{equation}
        In addition,
        \begin{equation}\label{contradition_useful_2}
            (1-\chi(2-\eta'(u)-\chi A'(u)))' = \chi (\eta''(u) + \chi A''(u)).
        \end{equation} 
        Because $\eta'' + \chi A'' > 0$ on $[\tilde u, v]$, it follows that
        $\eta(v) - \chi(v-A(v)) > 0$, contradicting \eqref{useful_contradiction_ineq}
        and completing the proof.
    \end{proof}

    \begin{corollary}
        For $\chi \geq 0$,
        \begin{equation}\label{eta_lowerbound}
            \eta(u) - \chi (u-A(u)) \geq 0 \quad \text{ on }  [0, 1].
        \end{equation}
    \end{corollary}
    We do not use \eqref{eta_lowerbound} elsewhere in this paper, but we include it
    in the case the reader might find it useful.
    \begin{proof}
        Since $\eta \geq 0$ by construction, the case $\chi = 0$ is immediate.
        Now we consider the case $\chi > 0$.
        For convenience, we set
        \begin{equation}
            \gamma(u) := \eta(u) - \chi (u-A(u)).
        \end{equation}
        By way of contradiction suppose at some
        point in $(0, 1)$, $\gamma(u) < 0$. Because
        $\gamma(0) = 0 = \gamma(1)$, there exists a point $u_0 \in (0, 1)$
        at which a negative local minimum occurs. Therefore, at $u_0$,
        \begin{equation}\label{gamma_local_min}
            \gamma(u_0) < 0 \quad \text{ and } \quad \gamma'(u_0) = 0.
        \end{equation}
        Using \eqref{rcl_pulled_identity},
        \[0 = \gamma'(u_0) = \eta'(u_0) - \chi(1-A'(u_0)) = 2- \chi - \frac{u_0 - A(u_0)}{\eta(u_0)}.\]
        Due to the first property of \eqref{gamma_local_min},
        \[-\frac{u_0 - A(u_0)}{\eta(u_0)} < -\frac{1}{\chi}.\]
        Therefore,
        \begin{equation}\label{eta_contradiction}
            0 < 2- \chi - \frac{1}{\chi} = \frac{-1}{\chi}(\chi - 1)^2 \leq 0.
        \end{equation}
        Hence, we arrive at a contradiction and conclude
        \eqref{eta_lowerbound}.
    \end{proof}

    Lastly, we give the proof of \Cref{concave_technical_lemma}, which is the key estimate in proving the weighted
    $L^1$-estimates in \Cref{pulled_adjoint_subsection}.
    \begin{proof}[Proof of {\Cref{concave_technical_lemma}}]
        By \Cref{asympotics_lemma}, if $\tilde u(t, x) < \frac{1}{100}$, then
        \begin{equation}
            2\eta(\tilde u)(\eta''(\tilde u) + \chi A''(\tilde u)) \leq -\frac{2}{\log\left(\tilde u\right)^2} + \frac{4\log\log\left(\frac{\theta}{\tilde u}\right)}{\log\left(\frac{\theta}{\tilde u}\right)^3} +  \frac{C}{\log\left(\frac{1}{\tilde u}\right)^3}. \label{precise_asympotics_expansion}
        \end{equation}
        In \eqref{precise_asympotics_expansion}, $\theta$ is a positive constant such that the minimal speed traveling wave $U_*$ has asymptotics
        \[U_*(x) \sim \theta x e^{-x}
        \quad \text{as } x \to \infty.\]
        Recall the preliminary upper bound on $\tilde u(t, x)$ in \eqref{pulled_L1_prelim}. This bound implies that
        there exists $N > 0$ so that $\tilde u (t, x) < \frac{1}{100}$ for all $x > N$, independent of $t$.
        Using \Cref{lowerbound_lemma}, there exists $C > 0$, so that for
        $t \geq 1$ and $x \geq 0$,
        \begin{equation}
            \tilde u(t, x) \geq Cx e^{-x}e^{-\frac{x^2}{4t}}. \label{lower_bound}
        \end{equation}
        Note a lower bound on $\tilde u$ is needed because $\sfrac{-1}{\log(\tilde u)^2}$ is decreasing in $\tilde u$ for $0 < \tilde u < 1$.

        First, insert \eqref{lower_bound} into the first term in the expansion \eqref{precise_asympotics_expansion} and compute
        \begin{equation}\label{good_term}
            \begin{split}
                -\frac{2}{\log(\tilde u(t, x))^2} &\leq \frac{-2}{\log(Cxe^{-x}e^{-\frac{x^2}{4t}})^2} = \frac{-2}{(x + \frac{x^2}{4t} - \log(Cx))^2} = \frac{-2}{x^2}\frac{1}{\left(1 + \frac{x}{4t} - \frac{\log(Cx)}{x}\right)^2}\\
                &\leq \frac{-2}{x^2}\left(1 + \frac{2\log(Cx)}{x} - \frac{x}{2t}\right)
                = \frac{-2}{x^2} - \frac{4\log(Cx)}{x^3} + \frac{1}{xt}
                \\&
                = \frac{-2}{x^2} - \frac{4\log(x)}{x^3} - \frac{4\log(C)}{x^3} + \frac{1}{xt}.
            \end{split}
        \end{equation}
        In the second inequality, we used that $-(1+z)^{-2} \leq -(1-2z)$ for $z \geq -1$,
        and that we can further increase $N$ so that
        \[-\frac{\log(Cx)}{x} \geq -1, \quad \text{for all } x \geq N.\]

        For $\tilde u$ small enough, the second term in \eqref{precise_asympotics_expansion} is increasing. Therefore, since
        $\tilde u \leq CU_*(x) \leq C (x+1) e^{-x}$, we can increase $N$ again if necessary so that for $x \geq N$,
        \begin{equation}\label{bad_term}
            \begin{split}
                \frac{4\log\log\left(\frac{\theta}{u}\right)}{\log\left(\frac{\theta}{u}\right)^3}
                &\leq \frac{4\log(\log(\theta) + x - \log(Cx))}{(x+ \log(\theta) - \log(Cx))^3}
            \leq \frac{4\log(x)}{x^3} \frac{1}{\left(1- \frac{\log(Cx)}{x} \right)^3}
            \\&
            \leq \frac{4\log(x)}{x^3}\left(1 + \frac{C\log(Cx)}{x}\right).
            \end{split}
        \end{equation}
        In the last inequality, we used that $(1-z)^{-3} \leq 1+Cz$, for $z \in [0, 1/2]$.

        Combining \eqref{good_term} and \eqref{bad_term}, we have that
        \begin{equation}
            2\eta(\tilde u)(\eta''(\tilde u) + \chi A''(\tilde u)) \leq \frac{-2}{x^2} + \frac{1}{xt} - \frac{4\log(C)}{x^3} + \frac{C\log(x)^2}{x^4} + \frac{C}{\log\left(\frac{1}{\tilde u}\right)^3}. \label{pulled_expansion}
        \end{equation}
        The last summand in~\eqref{pulled_expansion} can be bounded using the fact that by \eqref{pulled_L1_prelim}, $\tilde u(t, x) \leq Ce^{-x/2}$. Thus,
        \[2\eta(\tilde u)(\eta''(\tilde u) + \chi A''(\tilde u)) \leq \frac{-2}{x^2} + \frac{1}{xt} + \frac{C}{x^3}.\]
        The proof is complete.
    \end{proof}

    \section{Proof of the weighted Nash inequality}\label{weighted_nash_ineq_proof_section}
    
    One of our main tools in establishing \Cref{shape_defect_theorem} is the weighted Nash inequality \Cref{weighted_nash_ineq_lemma}.  We note that the standard, and, indeed, original argument to establish the (unweighted) Nash inequality is based on the Fourier transform~\cite{Nash}.  It is not clear how to adapt this in our setting given the half-space domain and the weighted $L^1$-term.  As a result, we develop a real-space approach instead.
    
    \begin{proof}[Proof of {\Cref{weighted_nash_ineq_lemma}}]
        By density, assume that $\varphi$ is smooth.  Let $L > 0$, $k$ be a nonnegative integer, and $a$ minimize $\varphi(x)$ over $[kL, (k+1)L]$.
        For a positive function $A(x)$ to be chosen later,
        \begin{equation}\label{nash_start}
            0 \leq \int_{kL}^a (x-kL) \left(\varphi_x A(x) + \frac{\varphi}{A(x)}\right)^2 dx = \int_{kL}^a (x-kL) \left(\varphi_x^2 A(x)^2 + \frac{\varphi^2}{A(x)^2}\right) dx + \int_{kL}^a (x-kL)(\varphi^2)_x dx.
        \end{equation}
        We integrate by parts the last integral in \eqref{nash_start} to find
        \begin{equation}\label{nash_ibp}
           \int_{kL}^a (x-kL)(\varphi^2)_x dx = (a-kL) \varphi(a)^2 - \int_{kL}^a \varphi^2 dx. 
        \end{equation}
        Using \eqref{nash_ibp} in \eqref{nash_start} and rearranging yields
        \begin{equation}\label{e.082401}
        	\int_{kL}^a \varphi^2 \left(1 - \frac{x-kL}{A(x)^2}\right) dx
			\leq \int_{kL}^a (x-kL) \varphi_x^2 A(x)^2 dx + (a-kL) \varphi(a)^2.
        \end{equation}
        We choose $A(x)^2 = 2(x-kL)$ so that the term on the left hand side of~\eqref{e.082401} becomes a pure $L^2$-type term with a positive coefficient.   Recall that $kL \leq a \leq (k+1)L$, so that $0\leq a-kL \leq L$. Hence,
        \begin{equation}
            \int_{kL}^a \varphi^2 dx \leq 4L^2 \int_{kL}^a \varphi_x^2 dx + 2L\varphi(a)^2. \label{left}
        \end{equation}

        On $[a, (k+1)L]$, we proceed in a similar way:
        \begin{align*}
            0 &\leq \int_a^{(k+1)L} ((k+1)L-x)\left(\varphi_x A(x) - \frac{\varphi}{A(x)}\right)^2 dx\\
            &= \int_a^{(k+1)L} ((k+1)L-x) \left(\varphi_x^2 A(x)^2 + \frac{\varphi^2}{A(x)^2}\right)dx - \int_a^{(k+1)L} ((k+1)L-x)(\varphi^2)_x dx\\
            &= \int_a^{(k+1)L} ((k+1)L-x) \left(\varphi_x^2 A(x)^2 + \frac{\varphi^2}{A(x)^2}\right)dx + ((k+1)L - a)\varphi(a)^2 - \int_a^{(k+1)L}\varphi^2 dx.
        \end{align*}
        In the last inequality, we integrated by parts.
        As before, we rearrange the above inequality to get
        \[\int_a^{(k+1)L} \varphi^2\left(1 - \frac{(k+1)L-x}{A(x)^2}\right)dx \leq \int_a^{(k+1)L} ((k+1)L-x) \varphi_x^2 A(x)^2 dx + ((k+1)L-a)\varphi(a)^2.\]
        Choosing $A(x)^2 = 2((k+1)L-x)$, we obtain
        \begin{equation}
            \int_a^{(k+1)L} \varphi^2dx \leq 4L^2\int_a^{(k+1)L} \varphi_x^2 dx + 2L\varphi(a)^2. \label{right}
        \end{equation}
        
        Summing \eqref{left} and \eqref{right} gives us
        \begin{equation}\label{e.middle of Nash}
        		\int_{kL}^{(k+1)L} \varphi^2dx \leq 4L^2\int_{kL}^{(k+1)L} \varphi_x^2 dx + 4L\varphi(a)^2.
        \end{equation}
        Except for the boundary term $4L\varphi(a)^2$, the other terms in \eqref{e.middle of Nash}
        are as desired.
	    To deal with the boundary term, note that, since $\beta \geq 0$,
        \[\int_{kL}^{(k+1)L} x^\beta dx \geq \frac{(k+1)^\beta L^{\beta + 1}}{C_\beta }.\]
        Therefore, we can write
        \[
        		L \varphi(a)^2
			\leq L \varphi(a)^2 \left(\frac{C_\beta}{(k+1)^\beta L^{\beta+1}} \int_{kL}^{(k+1)L}   x^\beta dx \right)^2
			= \frac{C_\beta^2}{(k+1)^{2\beta} L^{2\beta+1}} \left(\int_{kL}^{(k+1)L}  \varphi(a) x^\beta dx \right)^2.
        \]
        Recall that $a$ minimizes $\varphi$ over $[kL, (k+1)L]$; that is, $\varphi(a) \leq \varphi(x)$ on the domain of integration.  Using this in the above, we find
        \begin{align*}
            L \varphi(a)^2 
            	\leq \frac{C_\beta^2}{(k+1)^{2\beta} L^{2\beta+1}} \left(\int_{kL}^{(k+1)L}  \varphi(x) x^\beta dx \right)^2.
        \end{align*}
        Folding this into~\eqref{e.middle of Nash} and using that $(k+1)^{2\beta} \geq 1$ yields
        \[\int_{kL}^{(k+1)L} \varphi^2dx
        		\leq 4L^2\int_{kL}^{(k+1)L} \varphi_x^2 dx + \frac{C_\beta^2}{L^{2\beta + 1}}\left( \int_{kL}^{(k+1)L} \varphi x^\beta dx\right)^2.\]
        This inequality holds true for all $k$, so we now sum over each nonnegative integer $k$ to get
        \begin{align*}
            \int_0^\infty \varphi^2 dx &\leq 4L^2 \int_0^\infty \varphi_x^2 dx + \frac{C_\beta^2}{L^{2\beta + 1}} \sum_{k=0}^\infty\left(\int_{kL}^{(k+1)L} \varphi x^\beta dx\right)^2\\
            &\leq 4L^2 \int_0^\infty \varphi_x^2 dx + \frac{C_\beta^2}{L^{2\beta + 1}} \left(\int_0^\infty \varphi x^\beta dx\right)^2.
        \end{align*}
        The proof is complete.
    \end{proof}

    \section{Front location for $0 \leq \chi < 1$} \label{front_location_section}
    We now present the proof of \Cref{front_location_theorem}.\eqref{i.pulled_front}.
    The lower bound is immediate using that $u_0$ is steep \Cref{d.steepness} and comparing to the pure
    reaction-diffusion equation
    \[u_t = u_{xx} + u-A(u).\]
    This just leaves 
    the upper bound on $m(t)$, which can be done by proving an upper bound on $u$. Recall that $L_0$ is given in \eqref{initial_data_assumption},
    and let $L_1$ be a positive constant defined below in \eqref{convenient_lowerbound}.
    The following lemma gives such an upper bound on $u$.
    \begin{lemma}\label{upperbound_prop}
        Suppose $0 \leq \chi < 1$ and $u_0$ satisfies the conditions of \Cref{main_result_theorem}. Then there exists positive constants $C$ and $T$ so that
        for all $t > 0$,
        \begin{equation}\label{pulled_upperbound}
            u\left(t, x + 2t - \frac{3}{2}\log t\right) \leq \begin{cases}
                CU_*(x) + \frac{C}{(t+T)^{1-\eta}} & x \in (-\infty, B]\\
                CU_*(x) + \frac{C}{t^{1-\eta}}e^{-(1-\eta)x} & x \in\left[B, \frac{3}{2}\log(t + T) + \frac{3}{2}\log(t) +B\right]\\
                C\left(x- \frac{3}{2}\log t\right)e^{-x}e^{- \frac{\left(x-\frac{3}{2}\log t - B\right)^2}{4(t+T)}} & x\in  \left[\frac{3}{2}\log(t + T) + \frac{3}{2}\log(t) + B, \infty\right),
            \end{cases}
        \end{equation}
        with $B := 2T + L_0 + L_1$ and $\eta > 0$ that can be chosen arbitrarily small.
    \end{lemma}
    We first show how \Cref{upperbound_prop} implies \Cref{front_location_theorem}.
    \begin{proof}[Proof of {\Cref{front_location_theorem}}]
        Let $\lambda \in (0, 1)$ and $x_\lambda(t)$ be such that
        \begin{equation}
            u(t, x_\lambda(t)) = \lambda.
        \end{equation}
        Because the lower bound on the
        front is already known, it is sufficient to prove that there exists $M > 0$ such that for all $t > 0$
        \begin{equation}\label{front_sufficient}
            x_\lambda(t) + 2t - \frac{3}{2}\log t \leq M.
        \end{equation}

        Suppose by way of contradiction that $x_\lambda(t) + 2t - \frac{3}{2}\log t \to \infty$ as $t \to \infty$. If
        $x_\lambda(t) \leq \log(t + T) + \frac{3}{2}\log(t) + B$, we use the middle case in \eqref{pulled_upperbound}
        to find
        \[0 < \lambda = u\left(x_\lambda(t) + 2t - \frac{3}{2}\log t\right) \leq CU_*(x_\lambda(t)) + \frac{C}{(t+T)^{1-\eta}}e^{-(1-\eta)x_{\lambda}(t)} \xrightarrow{t \to \infty} 0.\]
        Therefore, this case is not possible.
        On the other hand if $x_\lambda(t) \geq \log(t + T) + \frac{3}{2}\log(t) + B$, then by the last case in \eqref{pulled_upperbound},
        \[0 < \lambda = u\left(x_\lambda(t) + 2t - \frac{3}{2}\log t\right) \leq \frac{C \log(t+T)}{t^{5/2}}\xrightarrow{t \to \infty} 0.\]
        Both cases lead to contradictions, so \eqref{front_sufficient} is true and
        completes the proof of \Cref{front_location_theorem}.
    \end{proof}

    The following result gives a useful estimate on $u$ shifted to the $2t-\frac{3}{2}\log t$ frame.
    \begin{corollary}\label{prelim_estimate_corollary}
        With the assumptions of \Cref{upperbound_prop}, there exists a $C > 0$ so that for all $x \in \R$,
        \begin{equation}\label{prelim_estimate_result}
            u\left(t, x + 2t - \frac{3}{2}\log t\right) \leq CU_*(x).
        \end{equation}
    \end{corollary}
    \Cref{prelim_estimate_corollary} follows immediately from \Cref{upperbound_prop} and the asymptotics \eqref{pulled_asympotics}
    of $U_*$. We omit the proof.

    To prove \Cref{upperbound_prop},
    we begin by constructing a supersolution on the positive half-line. 
    Following the ideas of \cite{shortproof}, we analyze the linearized Dirichlet problem.
    \begin{lemma} \label{supersolution_lemma}
        For $x \geq 2t$, let
        \begin{equation}\label{linearized_soln}
            \phi(t, x) = \frac{K(x-2t)}{t^{3/2}}e^{-(x-2t)}e^{\frac{-(x-2t)^2}{4t}},
        \end{equation}
        where $K$ is a positive constant.
        Then there exists a $T > 0$ such that
        \begin{equation}
            \bar \phi(t, x) = \left(1- \frac{1}{t^{1/4}}\right) \phi(t, x) \label{supersolution}
        \end{equation}
        is a nonnegative supersolution of \eqref{main_eq} for $t > T$.
    \end{lemma}

    \begin{proof}
        Consider the linearized equation
        \[\phi_t = \phi_{xx} + \phi.\] 
        Setting $\tilde \phi(t, x) = \phi(t, x+2t)$ and $\hat \phi(t, x) = e^x \tilde \phi(t, x)$ we have
        \begin{equation}
            \hat \phi_t = \hat \phi_{xx}. \label{linearized_pde}
        \end{equation}
        Then,
        \[\hat \phi(t, x) = \frac{Kx}{t^{3/2}}e^{\frac{-x^2}{4t}}\]
        solves \eqref{linearized_pde} with the Dirichlet boundary condition for any $K$.
        Set 
        \[\bar \phi(t, x) = \left(1- \frac{1}{t^{1/4}}\right)\phi(t, x) = C\left(1- \frac{1}{t^{1/4}}\right) \frac{(x-2t)}{t^{3/2}}e^{-(x-2t)}e^{\frac{-(x-2t)^2}{4t}}.\]
        A quick calculation shows
        \begin{align*}
            \bar \phi_t - \bar \phi_{xx} - \bar \phi + A(\bar \phi) + \chi A'(\bar \phi) \bar \phi_x &= \frac{1}{4t^{5/4}}\phi(t, x) + \left(1-\frac{1}{t^{1/4}}\right)(\phi_t - \phi_{xx} - \phi) + A(\bar \phi) + \chi A'(\bar \phi)\bar\phi_x\\
            &=\frac{1}{4t^{5/4}}\phi(t, x) + A(\bar \phi) + \chi A'(\bar \phi)\bar\phi_x\\
            &\geq \phi(t, x)\left(\frac{1}{4t^{5/4}} + C\phi_x(t, x)\right).
        \end{align*}
        Taking a derivative of $\phi$, we obtain
        \[\phi_x(t, x) = \frac{K e^{-(x-2t)}e^{-\frac{(x-2t)^2}{4t}}}{t^{3/2}}\left(1 - (x-2t) - \frac{(x-2t)^2}{2t}\right) \geq \frac{-C}{t^{3/2}}.\]
        Thus, for $t > T$ sufficiently large,
        \[\frac{1}{4t^{5/4}} + C\phi_x(x) \geq 0,\]
        which implies that $\bar \phi$ is a nonnegative supersolution for $x \geq 2t$.
    \end{proof}

    Before finding a supersolution for $x \leq 2t$, we first
    construct a modified supersolution from \eqref{supersolution} in order to facilitate matching
    with the supersolution on the left half-line.

    Define $\bar u_2$ as in \cite[Lemma 4.8]{giletti2022monostablepulledfrontslogarithmic},
    which was shown to be a nonnegative supersolution to \eqref{main_eq} for $x \geq 2t$ when $\chi = 0$.
    We wish to use $\bar u_2$ as a supersolution when $\chi > 0$, in which case we need an estimate on
    $\partial_x\bar u_2$. Parabolic estimates imply that for $t_0 > T > 1$ and $x_0 > 2t_0$,
    \begin{equation}\label{parabolic_estimate}
        \sup_{(t,x) \in Q_{1/2}(t_0, x_0)} |\partial_x \bar u_2(t, x)| \leq C\sup_{(t,x) \in Q_{1}(t_0, x_0)} |\bar u_2(t, x)|,
    \end{equation}
    where
    \[Q_r(t_0, x_0) = (t_0 - r^2, t_0) \times (x_0 - r, x_0 + r).\]
    We use the following estimates on $\bar u_2$. For
    $\epsilon \in (0, 1/4)$ and $x \in \left[2t - \frac{3+2\epsilon}{2} \log t, 2t + \sqrt t\right]$, we have
    \begin{equation}\label{compact_estimate}
        \bar u_2(t, x) \leq \frac{C\left(x-2t + \frac{3+2\epsilon}{2}\log t\right)}{t^{3/2}}e^{-(x-2t)}e^{-\frac{(x-2t + \frac{3+2\epsilon}{2}\log t)^2}{4t}} + \frac{C\left(x-2t + \frac{3+2\epsilon}{2}\log t\right)}{t^{5/2}}e^{-(x-2t)}.
    \end{equation}
    Moreover, for $x \in [2t + \sqrt t, \infty)$,
    \begin{equation}\label{large_x_estimate}
        \bar u_2(t, x) \leq C e^{-(x-2t)}.
    \end{equation}
    Applying \eqref{parabolic_estimate} to \eqref{compact_estimate} gives
    \[|\partial_x \bar u_2(t, x)| \leq \frac{C \log}{t^{3/2}} \qquad \text{ for } x \in [2t, 2t + \sqrt t].\]
    Similarly, \eqref{large_x_estimate} implies
    \[|\partial_x \bar u_2(t, x)| \leq Ce^{-(x-2t)} \leq Ce^{-\sqrt t} \qquad \text{ for }  x \in [2t + \sqrt t, \infty).\]
    In particular, for all $x \geq 2t$, one finds
    \[|\partial_x \bar u_2(t, x)| \leq \frac{C \log t}{t^{3/2}}.\]
    This bound is enough to control the advective term in \eqref{main_eq} and shows that
    $\bar u_2$ remains a supersolution for $\chi > 0$.

    We next match $\bar u_2$ with $\bar \phi$. At $x = 2t + \frac{3}{2}\log t$, we have
    \[\bar \phi\left(t, 2t + \frac{3}{2}\log t\right) \geq \frac{K\log t}{t^{3}},\]
    while
    \[\bar u_2\left(t, 2t + \frac{3}{2}\log t\right) \leq \frac{C\log t}{t^3} + o\left(\frac{\log t}{t^3}\right).\]
    Choosing $K$ large enough ensures that
    \[\bar u_2\left(t, 2t + \frac{3}{2}\log t\right) < \bar \phi\left(t, 2t + \frac{3}{2}\log t\right).\]
    Thus, for all $x \geq 2t$, both $\bar \phi$ and $\bar u_2$ are nonnegative supersolutions,
    so
    \begin{equation}\label{modified_super_soln}
        \bar u_3(t, x) = \begin{cases}
        \bar u_2(t, x) & x \in \left[2t, 2t + \frac{3}{2}\log t\right]\\
        \min \{\bar u_2(t, x), \bar \phi(t, x)\} & x\in \left[2t+\frac{3}{2}\log t, \infty\right)
    \end{cases}
    \end{equation}
    is also a nonnegative supersolution on $[2t, \infty)$.

    We now have a suitable supersolution for $x \geq 2t$, however, we need to connect $\bar u_3$ with a different supersolution
    for $x \leq 2t$. We construct the supersolution in the same way as
    Giletti in \cite[Lemma 4.7]{giletti2022monostablepulledfrontslogarithmic}. This result was originally proven in the absence of the
    advective term in \eqref{main_eq}. For completeness, we provide a proof of the result
    when the advective term is included.
    \subsection{Left half supersolution}
    The approach is to build this supersolution from the minimal speed traveling wave $U_*$. First we define
    some notation. Let $\delta \in (0, \frac{1}{2})$ and
    define $Z$ so that 
    \begin{equation}
        U_*(-Z) = 1-2\delta. \label{Z_def}
    \end{equation}
    We can choose $\delta$ small enough to assume that $Z > 0$ for $z \leq -Z$.
    Next, define $Z_\delta(t)$ such that
    \begin{equation}
        U_*(-Z_\delta(t)) = 1- \delta - \frac{\gamma}{t}. \label{Z_delta}
    \end{equation}
    Note that $(Z_\delta(t))_{t \geq 0}$ is a bounded, increasing sequence because $U_*$ is monotonically decreasing. Therefore,
    there exists a $Z_\delta^\infty > 0$ which is the limit of this sequence. In addition, since $U_*$ is continuous,
    $U_*(-Z_\delta^\infty) = 1-\delta$. Moreover, from \eqref{Z_delta}, it follows that
    \[Z_\delta'(t) = O\left(\frac{\gamma}{t^2}\right),\]
    which then implies that
    \begin{equation}\label{Z_closeness}
        Z_\delta(t) - Z_\delta^\infty = O\left(\frac{1}{t}\right).
    \end{equation}

    \begin{lemma}[Giletti] \label{left_half_supersolution_lemma}
        There is a positive supersolution $\bar u_1$ of \eqref{main_eq} for $t \geq T$ and $x < 2t - \frac{3}{2}\log t - Z_\delta(t) + Z_0(t)$.
        Moreover, $Z_0(t)$, defined below in \eqref{z_0_def}, satisfies for any $t \geq T$,
        \begin{equation}\label{supersolution_zero_pt}
            Z_0(t) > \frac{7}{2}\log t.
        \end{equation}
        In addition, for $x \leq 2t - \frac{3}{2}\log t - Z_\delta(t)$,
        \begin{equation}\label{left_left_supersolution}
            \bar u_1(t, x) = U_*\left(x-2t + \frac{3}{2}\log t\right) + \frac{\gamma}{t}
        \end{equation}
        while, for $x \geq 2t - \frac{3}{2}\log t - Z_\delta(t)$, 
        \begin{equation}\label{tw_closeness}
            \left|\bar u_1(t, x) - U_*\left(x - 2t + \frac{3}{2}\log t + Z_\delta(t) - Z_\delta^\infty\right)\right| \leq \left(\frac{3}{t}e^{-\left(x-2t+\sfrac{3}{2}\log t + Z_\delta(t)\right)}\right)^{1-\eta},
        \end{equation}
        with $\eta > 0$ arbitrarily small.
        Lastly, there exists $L_1 > 0$ such that
        \begin{equation}
            \bar u_1(T, x) \geq 1 \qquad \text{ for all } x \leq - L_1. \label{convenient_lowerbound}
        \end{equation}
    \end{lemma}

    With \Cref{left_half_supersolution_lemma}, we now have enough to prove \Cref{upperbound_prop}.
    \begin{proof}[Proof of {\Cref{upperbound_prop}}]
        Following the matching argument after the proof of Lemma 4.8 in \cite{giletti2022monostablepulledfrontslogarithmic}
        shows that
        \begin{equation}\label{complete_supersolution}
            \bar u(t, x) = \begin{cases}
                \bar u_1(t, x) & x \in (-\infty, 2t + \log t]\\
                \min\{\bar u_1(t, x), \bar u_3(t, x)\} & x \in \left[2t + \log t, 2t + \frac{3}{2}\log t\right]\\
                \bar u_3(t, x) & x \in \left[2t + \frac{3}{2}\log t, \infty\right]
            \end{cases}
        \end{equation}
        is a nonnegative supersolution over the whole real line for $t \geq T$.



        To apply the comparison principle, we just need to show that $\bar u(0, \cdot)$ sits above
        $u_0$.
        Recall that the assumptions on the initial data $u_0$ in \eqref{initial_data_assumption}, give us an $L_0$ so that
        $u_0(x) = 0$ for all $x \geq L_0$. Now consider $\bar u(T, x-(L_0 + L_1))$. We have
        \[u_0(x) = 0 \leq \bar u(T, x-(L_0 + L_1)) \qquad \text{for } x \geq L_0.\] 
        In addition, for $x \leq L_0$,
        $x-(L_0 + L_1) \leq -L_1$. So by \eqref{convenient_lowerbound}, 
        \[u_0(x) \leq 1 \leq \bar u_1(T, x-(L_0 + L_1)).\]
        Now we can apply the comparison principle, so that for any $t \geq 0$, and $x \in \R$,
        \[u(t, x) \leq \bar u(t+T, x-(L_0 + L_1)).\]

        Lastly, we prove \eqref{pulled_upperbound}.
        Shifting to $x \to x + 2t - \frac{3}{2}\log t$, we have
        \[u\left(t, x + 2t - \frac{3}{2}\log t\right) \leq \bar u\left(t+T, x+ 2t - \frac{3}{2}\log t-(L_0 + L_1)\right).\]
        For $x \leq 2T + \frac{3}{2}\log t - \frac{3}{2}\log (t+T) - Z_\delta(t+T) + L_0 + L_1$, we
        have by \eqref{left_left_supersolution} that
        \begin{equation}\label{left_left_estimate}
            \begin{split}
                u\left(t, x + 2t - \frac{3}{2}\log t\right) &\leq \bar u_1\left(t+T, x +2t - \frac{3}{2}\log t - (L_0 + L_1)\right)\\
                &= U_*\left(x-2T + \frac{3}{2}\log (t+T) - \frac{3}{2}\log t -(L_0+L_1)\right) + \frac{\gamma}{t+T}\\
                &= U_*(x-C) + \frac{\gamma}{t+T} \leq CU_*(x) + \frac{\gamma}{t+T}.
            \end{split}
        \end{equation}
        Recalling \eqref{tw_closeness}, yields for 
        \[
            x \in \left[2T + \frac{3}{2}\log t - \frac{3}{2}\log (t+T) - Z_\delta(t+T) + L_0 + L_1, 2T + \frac{3}{2}\log t+\log(t+T)+ (L_0 + L_1)\right], 
        \]
        the estimate
        \begin{equation}\label{left_middle_estimate}
            \begin{split}
                u\left(t, x + 2t - \frac{3}{2}\log t\right) &\leq U_*(x-C) + \left(\frac{3}{(t+T)} e^{-(x-2T + \frac{3}{2}\log(t+T) - \frac{3}{2}\log t + Z_\delta(t+T))}\right)^{1-\eta}\\
                &\leq C U_*(x) +\frac{Ce^{-(1-\eta)x}}{(t+T)^{1-\eta}}.
            \end{split}
        \end{equation}
        Combining \eqref{left_left_estimate} and \eqref{left_middle_estimate}, shows that
        for all $x \leq 2T +L_0+L_1$,
        \[
            u\left(t, x + 2t - \frac{3}{2}\log t\right) \leq  C U_*(x) +\frac{C}{(t+T)^{1-\eta}},
        \]
        which proves the first case in \eqref{pulled_upperbound}.

        For the second case, let 
        \[
            x \in \left[2T + \frac{3}{2}\log t+ \log(t+T)+ (L_0 + L_1), 2T + \frac{3}{2}\log(t) + \frac{3}{2}\log(t+T) + (L_0 + L_1)\right].
        \]
        By \eqref{complete_supersolution}, $\bar u \leq \min\{\bar u_1, \bar u_3\} \leq \bar u_1$ on this interval.
        Therefore,
        \[u\left(t, x + 2t - \frac{3}{2}\log t\right) \leq CU_*(x)  +\frac{Ce^{-(1-\eta)x}}{(t+T)^{1-\eta}}.\]
        This completes the proof of \Cref{upperbound_prop}.

        Finally, for $x \geq 2T + \frac{3}{2}\log(t + T) + \frac{3}{2}\log(t) + (L_0 + L_1)$,
        \begin{equation}\label{right_estimate}
            \begin{split}
                u\left(t, x + 2t - \frac{3}{2}\log t\right) &\leq \bar \phi\left(t+T, x + 2t - \frac{3}{2}\log t - (L_0+L_1)\right)\\
                &\leq C\left(x- \frac{3}{2}\log t\right)e^{-x}e^{-\frac{(x- \sfrac{3}{2}\log t -2T - L_0-L_1)^2}{4(t+T)}}.
            \end{split}
        \end{equation}
        This proves the third case in \eqref{pulled_upperbound}.
    \end{proof}

    We prove \Cref{left_half_supersolution_lemma} with a sequence of lemmas.
    The first of which gives a supersolution for $x$ ``far enough'' behind the front.
    \begin{lemma}\label{left_left_supersolution_lemma}
        Let
        \begin{equation}
            \bar u_1(t, x) = U_*\left(x-2t + \frac{3}{2} \log t\right) + \frac{\gamma}{t}. \label{super_soln_left}
        \end{equation}
        Then $\bar u_1$ is a positive supersolution of \eqref{main_eq} for $t > \frac{\gamma}{\delta}$, $x \leq 2t - \frac{3}{2}\log t - Z_\delta(t)$, and $Z_\delta$ given by \eqref{Z_delta}.
    \end{lemma}

    \begin{proof}
        First note that
        \[\bar u_1\left(t, 2t - \frac{3}{2}\log t - Z_\delta(t)\right) = U_*(-Z_\delta(t)) + \frac{\gamma}{t} = 1-\delta.\]
        In addition, since $t > \frac{\gamma}{\delta}$, we also have
        \begin{equation}\label{z_delta_bound}
            U_*(-Z_\delta(t)) = 1- \delta - \frac{\gamma}{t} > 1-2\delta = U_*(-Z).
        \end{equation}

        Define the operator $L$ by
        \[Lv := \partial_t v - \partial_{xx} v - v + A(v) + \chi A'(v) \partial_x v.\]
        By the mean value theorem, there exists $\xi_1(t,x)$ and $\xi_2(t,x)$ such that
        \[A(\bar u_1) = A\left(U_* + \frac{\gamma}{t}\right) = A(U_*) + A'(\xi_1)\frac{\gamma}{t}\]
        and
        \[A'(\bar u_1) = A'\left(U_* + \frac{\gamma}{t}\right) = A'(U_*) + A''(\xi_2)\frac{\gamma}{t},\]
        where $U_* \leq \xi_1, \xi_2 \leq U_* + \frac{\gamma}{t}$.
        Hence,
        \begin{align*}
            L\bar u_1 &= -2U_*' + \frac{3}{2t}U_*' - \frac{\gamma}{t^2} - U_*'' - U_* - \frac{\gamma}{t} + A(U_*) + A'(\xi_1)\frac{\gamma}{t} +  \chi U_*'\left(A'(U_*) + A''(\xi_2)\frac{\gamma}{t}\right)\\
            &= \frac{1}{t}\left[\frac{3}{2}U_*' + \gamma\left(-\frac{1}{t} - 1 + A'(\xi_1) + \chi A''(\xi_2)U_*'\right)\right]\\
            &\geq \frac{1}{t}\left[\frac{3}{2}U_*' + \gamma\left(-\frac{1}{t} - 1 + A'(U_*) + \norm{A''}_{L^\infty} U_*'\right)\right].
        \end{align*}

        We note that $U_*'(x) \to 0$ as $x \to -\infty$.
        Also, the assumptions \eqref{A_assumptions} on $A$ imply that $A'(u) > 1$ for
        $u$ sufficiently large. As an illustrating example,
        consider $A(u) = u^2$. In this case, $A'(u) = 2u$, so
        $A'(u) > 1$ whenever $u > 1/2$.
        For a general $A$, by choosing
        $\delta$ small enough and noting \eqref{z_delta_bound}, we have 
        \[-\frac{1}{t} - 1 + A'\left(U_*\right)+ \norm{A''}_{L^\infty} U_*' > C > 0, \quad \text{ for all } x < 2t - \frac{3}{2}\log t - Z_\delta(t).\]
        In the above inequality, $U_*$ and $U_*'$ are evaluated at $x-2t + \frac{3}{2}\log t$.
        Then, taking $\gamma$ large enough, we find that $L\bar u_1 \geq 0$.
        This proves \eqref{left_left_supersolution}.

        We can also now prove \eqref{convenient_lowerbound}.
        Let $T > \gamma$ and $s = 1 - \frac{\gamma}{T} \in (0, 1)$. Then we can find $L_1 > 0$ so that
        \begin{equation}
            U_*\left(- L_1 -2T + \frac{3}{2}\log T\right) = s.
        \end{equation}
        Since $U_*$ is decreasing, this implies that
        \begin{equation}
            U_*\left(x - 2T + \frac{3}{2}\log T\right) > s \quad \text{ for all } x < -L_1.
        \end{equation}
        Therefore,
        \begin{equation}
            \bar u_1(T, x) = U_*\left(x - 2T + \frac{3}{2}\log T\right) + \frac{\gamma}{T} > s + \frac{\gamma}{T} = 1 \quad \text{ for all } x < -L_1.
        \end{equation}
    \end{proof}

    Next, we extend $\bar u_1$ on the right of $2t - \frac{3}{2}\log t - Z_\delta(t)$ by setting
    \begin{equation}
        \bar u_1(t, x) = U_{3/2}\left(x-2t + \frac{3}{2}\log t + Z_\delta(t); t\right) \label{super_soln_middle}
    \end{equation}
    for $x \geq 2t - \frac{3}{2}\log t - Z_\delta(t)$, and $t$ large enough. For each $t$, $U_r(\cdot; t)$ is the solution to the ODE
    \[U_{3/2}'' + \left(2-\frac{3}{t}\right)U_{3/2}' + U_{3/2}-A(U_{3/2}) - \chi A'(U_{3/2})U_{3/2}' = 0\]
    on $[0, +\infty)$, with the boundary conditions
    \[U_{3/2}(0; t) = 1-\delta, \quad U_{3/2}'(0; t) = U_*'(-Z_\delta(t)).\]
    Since $2-\frac{3}{t}$ is below the minimum speed 2, eventually $U_{3/2}$ will become negative as $x \to \infty$. However, 
    we show $U_{3/2}$ remains positive
    long enough so that it intersects with our supersolution \eqref{supersolution} on the right. Let
    \begin{equation}\label{z_0_def}
        Z_0(t) = \inf \{x \geq 0: U_{3/2}(x; t) = 0\}.
    \end{equation}
    This is the $Z_0(t)$ mentioned in \Cref{supersolution_lemma}.

    \begin{lemma}
        The function $\bar u_1$ defined by \eqref{super_soln_left} and \eqref{super_soln_middle} is a positive supersolution
        of \eqref{main_eq} for all $t$ large enough and $x < 2t - \frac{3}{2}\log t - Z_\delta(t) + Z_0(t)$.
    \end{lemma}
    \begin{proof}
        The same proof as \cite[Lemma 4.5]{giletti2022monostablepulledfrontslogarithmic} gives the result.
    \end{proof}

    The last step is ensuring that $Z_0(t)$ is large enough so that our two supersolutions intersect. This is
    done through an approximation argument.

    \begin{lemma}\label{approx_lemma}
        Assume that $U$ and $U_\epsilon$ solve on the positive half line
        \begin{align*}
            U'' + 2U' + U-A(U) - \chi A'(U)U' = 0\\
            U''_\epsilon + (2-\epsilon)U'_\epsilon + U_\epsilon - A(U_\epsilon) - \chi A'(U_\epsilon)U_\epsilon' = 0,
        \end{align*}
        as well as 
        \[U(0) - U_\epsilon(0) = O(\epsilon), \quad U'(0) - U'_\epsilon(0) = O(\epsilon),\]
        as $\epsilon \to 0$. In addition assume
        \begin{equation}
            U, U' = O(ze^{-z}) \label{approx_assumption}
        \end{equation}
        as $z \to \infty$. Then for any $\eta > 0$, there exists $\epsilon_\eta$ such that for all $0 < \epsilon \leq \epsilon_\eta$ and
        $z \geq 0$,
        \begin{equation}
            |U(z)-U_\epsilon(z)| \leq (\epsilon e^{-z})^{1-\eta}. \label{approx_result}
        \end{equation}
    \end{lemma}

    Note that \eqref{approx_assumption} is satisfied for the minimal speed traveling wave $U_*$.
    \begin{proof}
        By stability theory for ODE \cite{ODE_stability}, we have that for any $Z > 0$, there exists $C_Z > 0$ such that
        \begin{equation}
            |U(z) - U_\epsilon(z)| + |U'(z) - U_\epsilon'(z)| \leq C_Z \epsilon \label{initial_approx_ineq}
        \end{equation}
        for all small $\epsilon > 0$ and $z \in [0, Z]$. In particular, we recover \eqref{approx_result} on $[0, Z]$
        as long as
        \begin{equation}\label{epsilon_eta_def}
            \epsilon \leq \epsilon_\eta \leq \left(\frac{e^{-(1-\eta)Z}}{C_Z}\right)^{1/\eta}.
        \end{equation}
        
        We aim to extend \eqref{approx_result} over $[Z, +\infty)$ with a barrier argument.
        Let
        \begin{equation}\label{V_defs}
            V = e^z U \qquad \text{ and } \qquad V_\epsilon = e^z U_\epsilon.
        \end{equation}
        Then the differential equations which $V$ and $V_\epsilon$ solve are
        \begin{align*}
            V'' &= e^z A(U) + \chi A'(U)(V'-V).\\
            V_\epsilon'' &= \epsilon(V_\epsilon' - V_\epsilon) + e^z A(U_\epsilon) + \chi A'(U_\epsilon)(V_\epsilon' - V_\epsilon).
        \end{align*}
        Let 
        \begin{equation}\label{big_W_def}
            W = V-V_\epsilon
        \end{equation}
        Then $W$ solves
        \begin{equation}\label{W_eqn}
            W'' = \epsilon(V_\epsilon-V_\epsilon') + e^z(A(e^{-z}V) - A(e^{-z}V_\epsilon)) + \chi (A'(e^{-z}V)(V'-V) - A'(e^{-z}V_\epsilon)(V_\epsilon' - V_\epsilon)).
        \end{equation}
        Notice that we can write
        \[\epsilon(V_\epsilon-V_\epsilon') = \epsilon(V-V') + \epsilon(W' - W)\]
        so that \eqref{W_eqn} becomes
        \begin{equation}\label{W_eqn_2}
            W'' = \epsilon(V-V') + \epsilon(W' - W) + e^z(A(e^{-z}V) - A(e^{-z}V_\epsilon)) + \chi (A'(e^{-z}V)(V'-V) - A'(e^{-z}V_\epsilon)(V_\epsilon' - V_\epsilon)).
        \end{equation}

        For the nonlinear terms in \eqref{W_eqn_2}, we apply Taylor's theorem to get an expression involving $W$.
        First we consider
        \begin{equation}\label{first_taylor}
            \begin{split}
                A(e^{-z}V_\epsilon) - A(e^{-z}V) &= e^{-z}(V_\epsilon-V)A'(e^{-z}V) + \frac{e^{-2z}(V_\epsilon-V)^2}{2}A''(\xi_1)\\
                &= -e^{-z}W (e^{-z}V A''(\xi_2)) + \frac{e^{-2z}W^2}{2}A''(\xi_1)\\
                &= e^{-2z} VW A''(\xi_2) + \frac{e^{-2z}W^2}{2}A''(\xi_1),
            \end{split}
        \end{equation}
        where $\xi_1$ is between $e^{-z} V$ and $e^{-z}V_\epsilon$ and $\xi_2$ between $0$ and $e^{-z}V$.
        Then for the terms involving $A'$ in \eqref{W_eqn_2}, we group them as
        \[A'(e^{-z}V)(V'-V) - A'(e^{-z}V_\epsilon)(V_\epsilon' - V_\epsilon) = (-A'(e^{-z}V)V + A'(e^{-z}V_\epsilon)V_\epsilon) + (A'(e^{-z}V)V' - A'(e^{-z}V_\epsilon)V_\epsilon').\]
        Since $W = V-V_\epsilon$, we can write $V_\epsilon = V - W$.
        Thus,
        \begin{equation}\label{second_taylor}
            \begin{split}
                -A'(e^{-z}V)V + A'(e^{-z}V_\epsilon)V_\epsilon &=  -A'(e^{-z}V)V + A'(e^{-z}V_\epsilon)V - A'(e^{-z}V_\epsilon)W\\
            &= V(e^{-z}(V_\epsilon-V)A''(\xi_3)) - W(e^{-z}V_\epsilon A''(\xi_4))\\
            &= -e^{-z} VW A''(\xi_3) - e^{-z}W(V-W)A''(\xi_4)\\
            &= -e^{-z} VW A''(\xi_3) - e^{-z}VWA''(\xi_4)+e^{-z}W^2 A''(\xi_4).
            \end{split}
        \end{equation}
        In this case, $\xi_3$ is between $e^{-z}V$ and $e^{-z}V_\epsilon$ and $\xi_4$ is between $e^{-z}V_\epsilon$ and $0$.
        Similarly,
        \begin{equation}\label{third_taylor}
            \begin{split}
                A'(e^{-z}V)V' - A'(e^{-z}V_\epsilon)V_\epsilon' &= A'(e^{-z}V)V' - A'(e^{-z}V_\epsilon)(V'-W')\\
                &= V'(A'(e^{-z}v) -A'(e^{-z}V_\epsilon)) + A'(e^{-z}V_\epsilon)W'\\
                &= e^{-z}V'WA''(\xi_5) + e^{-z}W'VA''(\xi_6) - e^{-z}W'W A''(\xi_6),
            \end{split}
        \end{equation}
        where $\xi_5$ is between $e^{-z}V$ and $e^{-z}V_\epsilon$ and $\xi_6$ is between $e^{-z}V_\epsilon$ and $0$.
        Inserting \eqref{first_taylor}, \eqref{second_taylor}, and \eqref{third_taylor} into \eqref{W_eqn_2}, we can write
        \begin{align*}
            W'' = \epsilon(V-V') + &\epsilon(W'-W) + e^{-z} VW \psi_1 + e^{-z}W^2\psi_2 + \chi e^{-z} (V'W \psi_3 + W'V \psi_4 - W'W \psi_5)
        \end{align*}
        where 
        \begin{equation}\label{psi_bounds}
            |\psi_i| \leq 3\norm{A''}_{L^\infty}
        \end{equation}
        
        At this point, we have manipulated the ODE for $W$ into a useful form.
        We employ a contradiction argument to show $W$ remains close to $0$.
        Consider the barrier
        \begin{equation}\label{barrier}
            \overline W(z) = \epsilon^{1-\eta}e^{\eta z}.
        \end{equation}
        Also, define $Z_\eta > 1$ such that for all $z \geq Z_\eta$,
        \begin{equation}
            Cze^{-\eta z} \leq \frac{\eta^2}{4}, \label{approx_lemma_contradiction}
        \end{equation}
        with $C$ to be chosen below.

        By \eqref{initial_approx_ineq} and its subsequent discussion, we can choose $\epsilon_\eta$ so that for all $\epsilon \leq \epsilon_\eta$ and for all $0 \leq z \leq Z_\eta$,
        \begin{equation}\label{barrier_bound}
            |W(z)| + |W'(z)| \leq C_{Z_\eta}\epsilon e^z < \eta \epsilon^{1-\eta}e^{\eta z}.
        \end{equation}
        We aim to show that
        \[-\overline W(z) < W(z) < \overline W(z)\]
        for all $z \geq Z_\eta$. Assume that this were not true and, without loss of generality, that
        $W$ intersects $\overline W$ at the leftmost contact point which we denote by $z_1$ on $(Z_\eta, +\infty)$.
        Note that by \eqref{barrier} and \eqref{barrier_bound}, we have
        \[W'(Z_\eta) < \overline W'(Z_\eta)\]
        Therefore, $\overline W - W$ reaches a positive maximum at some $z_0 \in (Z_\eta, z_1)$, implying
        $\overline W''(z_0) \leq W''(z_0)$. Consequently,
        \begin{align*}
            \overline W''(z_0) \leq W''(z_0) = \epsilon(V&-V')(z_0) + \epsilon(W'-W)(z_0) + e^{-z_0} V(z_0)W(z_0) \psi_1(z_0) + e^{-z_0}W(z_0)^2\psi_2(z_0)\\
            &+ \chi e^{-z_0} (V'(z_0)W(z_0) \psi_3(z_0) + W'(z_0)V(z_0) \psi_4(z_0) - W'(z_0)W(z_0) \psi_5(z_0)).
        \end{align*}
        At $z_0$, we also have that
        \[\overline{W}'(z_0) = W'(z_0), \quad |W(z_0)| \leq |\overline W(z_0)|.\]
        Recalling \eqref{psi_bounds}, we obtain the estimate
        \begin{align*}
            \overline W''(z_0) \leq \epsilon(|V|+|V'|)(z_0) + &\epsilon(\overline W'+\overline W)(z_0) + 3e^{-z_0}\norm{A''}_{L^\infty}\left( |V(z_0)|\overline W(z_0) + \overline W(z_0)^2\right)\\
            &+ 3\chi e^{-z_0} \norm{A''}_{L^\infty}(|V'(z_0)| \overline W(z_0) + \overline W'(z_0)|V(z_0)| + \overline W'(z_0)\overline W(z_0)).
        \end{align*}
        Due to \eqref{approx_assumption}, we can find $\bar C > 1$ so that for all $z > 1$,
        \begin{equation}\label{V_upperbound}
            |V(z)|+|V'(z)| = |e^z U(z)| + |e^z(U(z) + U'(z))| \leq 2|e^z U(z)| + |e^z U'(z)| \leq \bar Cz.
        \end{equation}
        Now we use \eqref{V_upperbound} to get 
        \begin{align*}
            \overline W''(z_0) \leq \bar C\epsilon z_0 + &\epsilon(\overline W'+\overline W)(z_0) + 3\bar Ce^{-z_0}\norm{A''}_{L^\infty}\left(z_0\overline W(z_0) + \overline W(z_0)^2\right)\\
            &+ 3\bar C\chi e^{-z_0} \norm{A''}_{L^\infty}(z_0 \overline W(z_0) + z_0\overline W'(z_0) + \overline W'(z_0)\overline W(z_0)).
        \end{align*} 
        Since $\overline W(z) = \epsilon^{1-\eta}e^{\eta z}$, we find
        \begin{equation}\label{approx_lemma_intermediate}
            \begin{split}
                \eta^2 \epsilon^{1-\eta}e^{\eta z_0} \leq \bar C\epsilon z_0 + &\epsilon^{2-\eta}e^{\eta z_0}(1+\eta) + 3\bar Ce^{-z_0}\norm{A''}_{L^\infty}\left( z_0\epsilon^{1-\eta}e^{\eta z_0} + \epsilon^{2-2\eta}e^{2\eta z_0}\right)\\
                &+3\bar C\chi e^{-z_0} \norm{A''}_{L^\infty}(z_0 \epsilon^{1-\eta}e^{\eta z_0} + z_0 \eta \epsilon^{1-\eta}e^{\eta z_0} + \eta \epsilon^{2-2\eta}e^{2\eta z_0}).
            \end{split}
        \end{equation}
        Divide each side of \eqref{approx_lemma_intermediate} by $e^{\eta z_0}$ so that
        \begin{equation}\label{approx_lemma_intermediate_2}
            \begin{split}
                \eta^2 \epsilon^{1-\eta} \leq \bar C\epsilon z_0e^{-\eta z_0} + &\epsilon^{2-\eta}(1+\eta) + 3\bar Ce^{-z_0}\norm{A''}_{L^\infty}\left( z_0\epsilon^{1-\eta} + \epsilon^{2-2\eta}e^{\eta z_0}\right)\\
                &+3\bar C\chi e^{-z_0} \norm{A''}_{L^\infty}(z_0 \epsilon^{1-\eta} + z_0 \eta \epsilon^{1-\eta} + \eta \epsilon^{2-2\eta}e^{\eta z_0}).
            \end{split}
        \end{equation}
        Since both $\chi, \eta < 1$, we can further simplify \eqref{approx_lemma_intermediate_2} to
        \begin{equation}
            \begin{split}
                \eta^2 \epsilon^{1-\eta} &\leq \bar C\epsilon z_0e^{-\eta z_0} + \epsilon^{2-\eta}(1+\eta) + 3\bar Ce^{-z_0}\norm{A''}_{L^\infty}\left( 3z_0\epsilon^{1-\eta} + 2 \epsilon^{2-2\eta}e^{\eta z_0}\right)\\
                &\leq \bar C\epsilon z_0e^{-\eta z_0} + \epsilon^{2-\eta}(1+\eta) + 3\bar Ce^{-z_0}\norm{A''}_{L^\infty}\left( 3z_0\epsilon^{1-\eta}e^{\eta z_0} + 3z_0 \epsilon^{1-\eta}e^{\eta z_0}\right)\\
                &= \bar C\epsilon z_0e^{-\eta z_0} + \epsilon^{2-\eta}(1+\eta) + 18\bar C\norm{A''}_{L^\infty}z_0 \epsilon^{1-\eta} e^{z_0(\eta - 1)}.
            \end{split}
        \end{equation}

        Because $\eta \in (0, 1/2)$, this means that $\eta - 1 \leq - \eta$. Therefore, using \eqref{approx_lemma_contradiction}
        with $C = 18\bar C\norm{A''}_{L^\infty}$,
        \[18\bar C\norm{A''}_{L^\infty}z_0 \epsilon^{1-\eta} e^{z_0(\eta - 1)} \leq 18\bar C\norm{A''}_{L^\infty} z_0 e^{-\eta z_0}\epsilon^{1-\eta} \leq \epsilon^{1-\eta} \frac{\eta^2}{4}.\]
        We finally have that
        \[\frac{3\eta^2 \epsilon^{1-\eta}}{4} \leq \epsilon \frac{\eta^2}{4} + \epsilon^{2-\eta}(1+\eta).\]
        The left hand side of the previous inequality has a smaller power of $\epsilon$ than
        each of the terms on the right; hence, by decreasing $\epsilon$, we arrive at a contradiction. We have shown that
        \[W(z) \leq \overline W(z) = \epsilon^{1-\eta}e^{\eta z}.\]
        A similar argument also shows that
        \[W(z) \geq -\overline W(z) = -\epsilon^{1-\eta}e^{\eta z}.\]
        Recalling the definition \eqref{big_W_def} for $W$ along with \eqref{V_defs},
        we conclude that
        \[|U(z) - U_\epsilon(z)| \leq \epsilon^{1-\eta} e^{(\eta - 1)z} = \left(\epsilon e^{-z}\right)^{1-\eta}. \]
        for all $z \geq 0$ and any $\eta > 0$.
    \end{proof}

    We are now equipped to complete the proof of \Cref{left_half_supersolution_lemma}.

    \begin{proof}[Proof of {\Cref{left_half_supersolution_lemma}}]
        By \Cref{approx_lemma} with $\epsilon = \frac{3}{t}$, $U = U_*(\cdot - Z_\delta^\infty)$, and $U_\epsilon = U_{3/2}(\cdot; t)$,
    we have that, for any $\eta > 0$ and $z \geq 0$,
    \begin{equation}
        |U_*(z-Z_\delta^\infty) - U_{3/2}(z; t)| \leq \left(\frac{3}{t}e^{-z}\right)^{1-\eta}. \label{pulled_approx_result}
    \end{equation}
    Recall that, for $x \geq 2t - \frac{3}{2}\log t - Z_\delta(t)$,
    \begin{equation}
        \bar u_1(t, x) = U_{3/2}\left(x - 2t + \frac{3}{2}\log t + Z_\delta(t); t\right). \label{approximate_tw}
    \end{equation}
    Therefore, setting $z = x-2t + \frac{3}{2}\log t + Z_\delta(t)$,
    \begin{equation}
        \begin{split}
            \left|\bar u_1(t, x) - U_*\left(x - 2t + \frac{3}{2}\log t + Z_\delta(t) - Z_\delta^\infty\right)\right| &= \left|U_{3/2}(z; t) - U_*(z- Z_\delta^\infty)\right|\\
            &\leq \left(\frac{3}{t}e^{-z}\right)^{1-\eta} = \left(\frac{3}{t}e^{-(x - 2t + \frac{3}{2}\log t + Z_\delta(t))}\right)^{1-\eta}.
        \end{split}
    \end{equation}
    This proves \eqref{tw_closeness} in \Cref{left_half_supersolution_lemma}.

    Lastly, we show that $U_{3/2}(z; t)$ remains positive for $z$ large enough. Since $U_*$ is always positive, we may then assume
    without loss of generality that $U_{3/2}(z; t) \leq U_*(z-Z_\delta^\infty)$. Therefore by \eqref{pulled_approx_result},
    \[U_*(z-Z_\delta^\infty) - \frac{3^{1-\eta}}{t^{1-\eta}}e^{-(1-\eta)z} \leq U_{3/2}(z; t).\]
    We note that the previous inequality implies for any $M > 0$, there exists a $T$ large enough so that
    \begin{equation}\label{constant_lowerbound}
        U_*(M-Z_\delta^\infty) > 0 \quad \text{ for } t > T.
    \end{equation}
    Using the pulled asymptotics \eqref{pulled_asympotics} of $U_*(z)$ as $z \to \infty$, we have
    \[U_{3/2}(z; t) \geq C (z-Z_{\delta}^\infty)e^{-(z-Z_{\delta}^\infty)} - \frac{3^{1-\eta}}{t^{1-\eta}}e^{-(1-\eta)z}.\]
    By \eqref{constant_lowerbound}, we have for all $t > T$, $U_{3/2}(z;t)$ is positive up to the $z > 0$ at which $C (z-Z_\delta^\infty) = 3^{1-\eta}$.
    Therefore, we consider when $C (z-Z_\delta^\infty) \geq 3^{1-\eta}$. In this case,
    \begin{align*}
        U_{3/2}(z; t) \geq 3^{1-\eta}e^{-z}\left(e^{z_\delta^\infty} - \frac{1}{t^{1-\eta}}e^{\eta z}\right) \geq 3^{1-\eta}e^{-z}\left(1 - \frac{1}{t^{1-\eta}}e^{\eta z}\right).
    \end{align*}
    Hence, for all $z \leq \frac{1-\eta}{\eta}\log t$, $U_{3/2}(z; t) \geq 0$. Since
    $\eta$ can be arbitrarily small, we can choose $\eta$ such
    that $U_{3/2}(z; t) \geq 0$ if $z \leq \frac{7}{2}\log t$. This proves \eqref{supersolution_zero_pt} and completes the proof of \Cref{left_half_supersolution_lemma}.
    \end{proof}

    \appendix

    \section{Asymptotics for the traveling wave profile function in the pulled case} \label{eta_asympotics}
    In \cite[Lemma 3.4]{an2023locationdeterminesconvergencerate}, the asymptotics for $\eta(u), \eta'(u),$ and $\eta''(u)\eta(u)$ are
    stated. We compute the asymptotics to one more order here. Specifically, we are interested in the asymptotics for
    $\eta(u)(\eta''(u) + \chi A''(u))$, which are needed for the weighted $L^1$-calculations in
    \Cref{pulled_adjoint_subsection}.

    Since $0 \leq \chi < 1$, the minimal speed
    traveling wave (up to the appropriate shift) has the asymptotics
    \begin{equation}
        U_*(x) = \theta x e^{-x} + O\left(e^{-(1+\epsilon)x}\right), \quad \text{as } x \to +\infty, \label{pulled_asymptotics}
    \end{equation}
    where $\epsilon, \theta  > 0$.

    \begin{lemma} \label{asympotics_lemma}
        Suppose $0 \leq \chi < 1$ so that $U_*$ has the asymptotics given by \eqref{pulled_asymptotics} as $x \to +\infty$. There exists
        $C > 0$, so that for all $u \in (0, 1/100)$,
        \begin{equation}
            \eta(u)(\eta''(u) + \chi A''(u)) \leq -\frac{1}{\log\left(u\right)^2} + \frac{2\log\log\left(\frac{\theta }{u}\right)}{\log\left(\frac{\theta }{u}\right)^3}  + \frac{C}{\log\left(\frac{1}{u}\right)^3}. \label{asympotics_lemma_result}
        \end{equation}
    \end{lemma}

    \begin{proof}
        By \eqref{eta_reduction_identity}, we can write
        \[\eta(u)(\eta''(u) + \chi A''(u)) = -(1-\eta'(u))^2 + A'(u)(1-\chi\eta'(u)).\]
        With the assumptions on $A(u)$, we have

        \begin{equation}
            \eta(u)(\eta''(u) + \chi A''(u)) = -(1-\eta'(u))^2 + O(u). \label{expand_term}
        \end{equation}
        Therefore, we reduce the problem to computing the asymptotics for $\eta'(u)$.

        Take $u \in (0, 1/100)$, and define $x_u$ so that $U_*(x_u) = u$. Since $U_*$ is
        strictly decreasing, $x_u$ is unique, and $x_u = U_*^{-1}(u)$. Then, as $u \to 0^+$,
        \begin{equation}
            x_u = \log\left(\frac{\theta }{u}\right) + \log\log\left(\frac{\theta }{u}\right) + O\left(\frac{\log\log\left(\frac{\theta }{u}\right)}{\log\left(\frac{\theta }{u}\right)}\right).
        \end{equation}
        Now we compute, using the definition of $\eta(u)$, that
        \begin{equation}\label{eta_deriv}
            \eta'(u) = \frac{d}{du} \eta(U_*(x_u)) = -\frac{d}{du} U_*'(x_u) = - U_*''(x_u) \frac{dx_u}{du}.
        \end{equation}
        Since $U_* \in C^2(\R)$, we have that
        \begin{equation}\label{inverse_deriv}
            \begin{split}
                \frac{dx_u}{du} &= \frac{d}{du}(U_*^{-1}(u)) = \frac{1}{U_*'(x_u)} = \frac{1}{-\theta x_u e^{-x_u} + \theta e^{-x_u} + O\left(e^{-(1+\epsilon)x_u}\right)}\\
            &= \frac{1}{-U_*(x_u)\left(1 - \frac{1}{x_u} + O\left(x_u^{-1} e^{-\epsilon x_u}\right)\right)} = \frac{-1}{u}\left(1+ \frac{1}{x_u} + O(x_u^{-2})\right).
            \end{split}
        \end{equation}
        In addition,
        \begin{equation}\label{second_deriv_asympotics}
            \begin{split}
                -U''(x_u) &= -\theta x e^{-x_u} + 2\theta e^{-x_u} + O\left(e^{(-1+\epsilon)}x_u\right) = -U_*(x_u)\left(1- \frac{2}{x_u} + O\left(e^{-\epsilon x_u}\right)\right)\\
            &= -u\left(1- \frac{2}{x_u} + O(e^{-\epsilon x_u})\right).
            \end{split}
        \end{equation}
        Substituting \eqref{inverse_deriv} and \eqref{second_deriv_asympotics} into \eqref{eta_deriv} yields
        \begin{align*}
            \eta'(u) &= \left(1- \frac{2}{x_u} + O(e^{-\epsilon x_u})\right)\left(1+ \frac{1}{x_u} + O(x_u^{-2})\right) = 1 - \frac{1}{x_u} + O(x_u^{-2})\\
            &= 1 - \frac{1}{\log\left(\frac{\theta }{u}\right) + \log\log\left(\frac{\theta }{u}\right) + O\left(\frac{\log\log\left(\frac{\theta }{u}\right)}{\log\left(\frac{\theta }{u}\right)}\right)} + O(x_u^{-2})\\
            &= 1 - \frac{1}{\log\left(\frac{\theta }{u}\right)} \frac{1}{1 + \frac{\log\log\left(\frac{\theta }{u}\right)}{\log\left(\frac{\theta }{u}\right)} + O\left(\frac{\log\log\left(\frac{\theta }{u}\right)}{\log\left(\frac{\theta }{u}\right)^2}\right)}+ O(x_u^{-2}) \\
            &= 1 - \frac{1}{\log\left(\frac{\theta }{u}\right)} \left(1 - \frac{\log\log\left(\frac{\theta }{u}\right)}{\log\left(\frac{\theta }{u}\right)} + O\left(\frac{\log\log\left(\frac{\theta }{u}\right)}{\log\left(\frac{\theta }{u}\right)^2} \right) \right).
        \end{align*}

        Inserting the previous expansion into \eqref{expand_term}, we find
        \begin{align*}
            \eta(u)(\eta''(u) + \chi A''(u)) &= -(1-\eta'(u))^2 + O(u)\\
            &= -\frac{1}{\log\left(\frac{\theta }{u}\right)^2} \left(1 - \frac{\log\log\left(\frac{\theta }{u}\right)}{\log\left(\frac{\theta }{u}\right)} + O\left(\frac{\log\log\left(\frac{\theta }{u}\right)}{\log\left(\frac{\theta }{u}\right)^2} \right) \right)^2 + O(u)\\
            &= -\frac{1}{\log\left(\frac{\theta }{u}\right)^2} + \frac{2\log\log\left(\frac{\theta }{u}\right)}{\log\left(\frac{\theta }{u}\right)^3} + O\left(\frac{\log\log\left(\frac{\theta }{u}\right)}{\log\left(\frac{\theta }{u}\right)^4}\right).
        \end{align*}

        Lastly, since $(1-z)^{-2} \geq 1+2z$ for $z \leq 1$, we can write
        \[-\frac{1}{\log\left(\frac{\theta }{u}\right)^2} = -\frac{1}{\log(u)^2}\frac{1}{\left(1- \frac{\log (\theta )}{\log(u)}\right)^2} \leq -\frac{1}{\log(u)^2}\left(1 + \frac{2\log(\theta )}{\log(u)}\right).\]
        Hence,
        \[\eta(u)(\eta''(u) + \chi A''(u)) \leq -\frac{1}{\log\left(u\right)^2} + \frac{2\log\log\left(\frac{\theta }{u}\right)}{\log\left(\frac{\theta }{u}\right)^3} + \frac{C}{\log\left(\frac{1}{u}\right)^3}.\]
    \end{proof}

    \section{Lower bound for $u$ in the pulled case} \label{lower_bound_section}
    Recall the notation
    \[\tilde u(t, x) = u\left(t, x + 2t - \frac{3}{2}\log t\right)\]
    In \cite[Proposition 3.1]{shortproof}, a lower bound of
    \begin{equation}
        \tilde u(t, x) \geq \kappa_{\sigma} xe^{-x} \label{hnrr_lower_bound}
    \end{equation}
    was proven for all $t \geq 1$ and $0\leq y \leq \sigma\sqrt t$.
    By a simple adjustment of the proof, we give a related lower bound that extends \eqref{hnrr_lower_bound} to the whole positive half line.
    \begin{lemma}\label{lowerbound_lemma}
        Suppose $u$ satisfies \eqref{initial_data_assumption} and $w_0 = -u_0'(x) - \eta(u_0) \geq 0$. Then there exists $C > 0$ so that for all $x \geq 0$,
        \begin{equation}
            \tilde u(t, x) \geq Cxe^{-x}e^{-\frac{x^2}{4t}}. \label{lower_bound_ineq}
        \end{equation}
    \end{lemma}
    \begin{proof}
        Suppose $\phi$ solves the linearized problem:
        \[\phi_t = \phi_{xx} + \phi.\]
        Then, $\bar \phi(t, x) = e^x \tilde \phi(t, x) = e^x\phi(t, x + 2t)$ solves
        \[\bar \phi = \bar \phi_{xx}.\]
        In addition, suppose that $\bar \phi$ has the Dirichlet boundary condition at $0$, so that
        \[\bar \phi(t, 0) = 0, \quad \text{for all } t > 0.\]
        Therefore, for an initial data $\bar \phi_0: \R_+ \to \R$, we have that
        \begin{equation}
            \tilde \phi(t, x) = \frac{e^{-x}}{\sqrt{4\pi t}}\int_0^\infty \left(e^{-\frac{(x-z)^2}{4t}} - e^{- \frac{(x+z)^2}{4t}}\right)\bar\phi_0(z)dz. \label{greens_solution}
        \end{equation}

        Since $w_0 \leq 0$, this means that $u_x \leq 0$ for all $t \geq 0$. Hence, the solution $v$ to the equation
        \begin{equation}
            v_t = v_{xx} + v - A(v) + 2v_x \label{subsolution_eqn}
        \end{equation}
        with initial data $u_0$, satisfies $v(t, x) \leq u(t, x+2t)$ by the parabolic comparison principle.
        Although, $\tilde \phi$ is not immediately a subsolution of \eqref{subsolution_eqn}, 
        following \cite{shortproof}, there exists a uniformly bounded function $a(t)$ so that for all $t \geq 0$, $0 < \theta  \leq a(t) \leq a_1 < \infty$, and
        $a(t) \tilde \phi(t, x)$ is a subsolution of \eqref{subsolution_eqn}.
        By the assumption \eqref{initial_data_assumption} on the initial data and since $u_0$ is monotonically decreasing, we
        can assume $u_0(x) > 0$ for $x < L_0$. Also, we can assume $L_0 \leq 0$ because if it were not the case, then
        we could consider $x < - L_0$.
        
        In particular, there exists $M > 0$ such that
        \begin{equation}\label{initial_data_lb}
            u_0(x) > M \qquad \text{ for } x < L_0 - 1. 
        \end{equation}
        
        Take a smooth initial data $\tilde \phi(0, x) \leq M$ satisfying
        \begin{equation}
            \tilde \phi(0, 0) = 0, \qquad \tilde \phi(0, x) = M \quad \text{ on } [1, 2], \qquad\text{ and } \tilde \phi(0, x) = 0 \quad \text{ on } [3, \infty). \label{chosen_data}
        \end{equation}
        Then by \eqref{initial_data_lb}, for $x \leq 3$, $u_0(x + (L_0 - 4)) > M$.
        In particular, this means $\tilde \phi(0, x) \leq u_0(x + (L_0 - 4))$ on $[0, \infty)$. In addition, for all $t \geq 0$, $\tilde \phi(t, 0) = 0 \leq v(t, L_0 - 4)$.
        Then by the comparison principle, we have, for $x \geq 0$,
        \[\frac{a(t)}{a_1}\tilde \phi(t, x) \leq v(t, x + (L_0 - 4)) \leq u(t, x + 2t + (L_0 - 4)).\]

        Now we show that \eqref{lower_bound_ineq} holds. Note that \eqref{greens_solution} can be written as
        \[\tilde \phi(t, x) = \frac{e^{-x}e^{-\frac{x^2}{4t}}}{\sqrt{4\pi t}} \int_0^\infty e^{-\frac{z^2}{4t}}\left(e^{\frac{xz}{2t}} - e^{-\frac{xz}{2t}}\right)\bar\phi_0(z)dz.\]
        By the Mean Value Theorem, there exists a $\xi \in \left(0, \frac{xz}{2t}\right)$, so that
        \[e^{\frac{xz}{2t}} = 1 + \frac{xz}{2t}e^\xi \geq 1 + \frac{xz}{2t}.\]
        Hence,
        \begin{align*}
           \tilde \phi(t, x) \geq \frac{e^{-x}e^{-\frac{x^2}{4t}}}{\sqrt{4\pi t}} \int_0^\infty e^{-\frac{z^2}{4t}}\left(1 + \frac{xz}{2t} - e^{-\frac{xz}{2t}}\right)\bar\phi_0(z)dz \geq \frac{xe^{-x}e^{-\frac{x^2}{4t}}}{2t\sqrt{4\pi t}} \int_0^\infty ze^{-\frac{z^2}{4t}}\bar \phi_0(z)dz.
        \end{align*}
        In the second inequality we used that $x, z \geq 0$, so
        \[1-e^{-\frac{xz}{2t}} \geq 0.\]
        The integrand is increasing with respect to $t$, so, for all $t \geq 1$,
        \[\tilde \phi(t, x) \geq \frac{xe^{-x}e^{-\frac{x^2}{4t}}}{2t\sqrt{4\pi t}} \int_0^\infty ze^{-\frac{z^2}{4}}\bar \phi_0(z)dz.\]
        Since $\bar\phi_0(z) = e^z \tilde \phi(0, z)$ and using \eqref{chosen_data}, we have
        \[\tilde \phi(t, x) \geq \frac{xe^{-x}e^{-\frac{x^2}{4t}}}{2t\sqrt{4\pi t}} \int_{1}^{2} ze^{-\frac{z^2}{4}} Me^z dz = \frac{Cxe^{-x}e^{-\frac{x^2}{4t}}}{t^{3/2}}.\]
        Thus, for $x \geq L_0 - 4 + \frac{3}{2}\log t$,
        \begin{equation}\label{shifted_lowerbound}
            \begin{split}
                \tilde u(t, x) \geq \frac{\theta }{a_1}\tilde \phi\left(t, x -(L_0-4) - \frac{3}{2}\log t\right) &\geq C\left(x - \frac{3}{2}\log t -(L_0-4)\right)e^{L_0-4}e^{-x}e^{-\frac{(x- \frac{3}{2}\log t -(L_0-4))^2}{4t}}\\
                &\geq C\left(x - \frac{3}{2}\log t\right)e^{-x}e^{-\frac{(x- \frac{3}{2}\log t)^2}{4t}}.
            \end{split}
        \end{equation}
        In the last inequality, we used that $L_0 \leq 0$.
        Combining \eqref{hnrr_lower_bound} with \eqref{shifted_lowerbound} implies that, for all $t \geq 1$ and all $x \geq 0$,
        \[\tilde u(t, x) \geq Cxe^{-x}e^{-\frac{x^2}{4t}}.\]
    \end{proof}

    \section{Derivation of the shape defect function PDE} \label{derivation}
    Recall the definition \eqref{shape_defect_function} of the shape defect function.
    In the following calculations, we repeatedly use
    the following identity relating $u, A(u), A'(u), \eta(u)$ and $\eta'(u)$ that was shown in \eqref{rcl_pulled_identity}:
    \begin{equation}\label{key_identity}
        u-A(u) = \eta(u)(2-\eta'(u)) - \chi A'(u)\eta(u).
    \end{equation}
    With \eqref{key_identity} in mind along with the PDE \eqref{main_eq} for $u$, we compute
    \begin{align*}
        w_t - w_{xx} &= (-u_x-\eta(u))_t +(u_x +\eta(u))_{xx} = -u_{tx} - \eta'(u)u_t + u_{xxx} + \eta''(u)u_x^2 + \eta'(u)u_{xx}\\
        &= -(u_{xx}-\chi A'(u)u_x + u-A(u))_x\\
        &\hspace*{0.75cm} - \eta'(u)(u_{xx}-\chi A'(u)u_x + u-A(u)) + u_{xxx} + \eta''(u)u_x^2 + \eta'(u)u_{xx}\\
        &= \chi A''(u)u_x^2 + \chi A'(u)u_{xx} -(\eta(u)(2-\eta'(u)) -\chi A'(u)\eta(u))_x\\
        &\hspace*{0.75cm} +\chi A'(u)\eta'(u)u_x -\eta'(u)\eta(u)(2-\eta'(u)) +\chi A'(u)\eta'(u)\eta(u) + \eta''(u)u_x^2\\
        &= \chi A''(u)u_x^2 + \chi A'(u)u_{xx} - \eta'(u)(2-\eta'(u))u_x + \eta(u)\eta''(u)u_x\\
        &\hspace*{0.75cm} + \chi A''(u)\eta(u)u_x + \chi A'(u)\eta'(u)u_x\\
        &\hspace*{0.75cm} +\chi A'(u)\eta'(u)u_x -\eta'(u)\eta(u)(2-\eta'(u)) +\chi A'(u)\eta'(u)\eta(u) + \eta''(u)u_x^2.\\
    \end{align*}
    Note that we can write
    \[\chi A'(u)u_{xx} + \chi A'(u)\eta'(u)u_x = \chi A'(u) (u_{xx} + \eta'(u)u_x) = -\chi A'(u)w_x.\]
    Similarly,
    \begin{align*}
        -\eta'(u)(2-\eta'(u))u_x - \eta'(u)\eta(u)(2-\eta'(u)) &= -\eta'(u)(2-\eta'(u))(u_x + \eta(u))= \eta'(u)(2-\eta'(u))w.
    \end{align*}
    Therefore,
    \begin{align*}
        w_t - w_{xx} &= - \chi A'(u)w_x + \eta'(u)(2-\eta'(u))w + \chi A''(u)u_x^2\\
        &\hspace*{0.75cm}+ \eta(u)\eta''(u)u_x + \chi A''(u)\eta(u)u_x + \chi A'(u)\eta'(u)u_x\\
        &\hspace*{0.75cm}+ \chi A'(u)\eta'(u)\eta(u)+\eta''(u)u_x^2.
    \end{align*}
    In addition, we have
    \begin{align*}
        \eta(u)\eta''(u)u_x + \eta''(u)u_x^2 = \eta''(u)u_x (u_x + \eta(u)) = -\eta''(u)u_x w = -\eta''(u) w(-w - \eta(u)) = w\eta''(u)(w + \eta(u)),
    \end{align*}
    along with
    \begin{align*}
        \chi A''(u) u_x^2 + \chi A''(u)\eta(u)u_x = \chi A''(u)u_x (u_x + \eta(u)) = \chi w A''(u)(w+\eta(u)),
    \end{align*}
    and
    \begin{align*}
        \chi A'(u)\eta'(u)u_x + \chi A'(u)\eta'(u)\eta(u) = \chi A'(u)\eta'(u)(u_x + \eta(u)) = -\chi A'(u) \eta'(u)w.
    \end{align*}
    Combining the previous calculations, we obtain
    \begin{align*}
        w_t - w_{xx} + \chi A'(u) w_x = \eta'(u)(2-\eta'(u))w + w\eta''(u)(w+\eta(u)) + \chi w A''(u)(w+\eta(u)) - \chi A'(u)\eta'(u)w.
    \end{align*}
    Therefore,
    \begin{equation}
        w_t = w_{xx} - \chi A'(u)w_x + w\eta'(u)[2-\eta'(u) - \chi A'(u)] + w(w+\eta(u))(\eta''(u) + \chi A''(u)).
    \end{equation}
    We point out that, by \eqref{key_identity}, one finds
    \begin{equation}\label{rearranged_identity}
        \begin{split}
            \eta''(u) + \chi A''(u) &= \left(\frac{A(u) - u}{\eta(u)}\right)' = \frac{A'(u)-1}{\eta(u)} - \frac{\eta'(u)(A(u)-u)}{\eta(u)^2}\\
        &= \frac{A'(u)-1}{\eta(u)} - \frac{\eta'(u)(\chi A'(u) + \eta'(u)-2)}{\eta(u)}.
        \end{split}
    \end{equation}
    Rearranging \eqref{rearranged_identity} to solve for $\eta'(u)(2-\eta'(u) - \chi A'(u))$, we obtain
    \[w_t = w_{xx} - \chi A'(u)w_x + (1 - A'(\tilde u))w + 2w(w+\eta(u))(\eta''(u) + \chi A''(u)).\]

    In the special case when $\chi = 1$, the traveling wave profile function, $\eta$, is explicitly given by
    \[\eta(u) = u-A(u).\]
    Noting that $\eta''(u) = -A''(u)$, this allows us to simplify the equation for $w$ to
    \[w_t = w_{xx} - A'(u)w_x + (1-A'(u))w.\]

\bibliographystyle{plain}
\bibliography{references}

\end{document}